\documentclass[10pt, a4paper]{article}
\usepackage[utf8]{inputenc}

\usepackage[T1]{fontenc} 
\usepackage[english]{babel}

\usepackage[usenames,dvipsnames]{pstricks}
\usepackage{epsfig}
\usepackage{pst-grad} 
\usepackage{pst-plot} 

\usepackage{euscript}
\usepackage{amsmath}
\usepackage{amsfonts}
\usepackage{amssymb}
\usepackage{float}
\usepackage{graphicx}

\usepackage{hyperref}
\usepackage{xcolor}
\hypersetup{linkcolor=blue}
\hypersetup{colorlinks=true}

\newcommand{\skipline}{\vspace{0.6cm}}

\numberwithin{equation}{section}

\newcommand{\R}{\mathbb{R}}
\newcommand{\Z}{\mathbb{Z}}

\usepackage{amsthm}
\theoremstyle{plain}
\newtheorem{definition}{Definition}
\newtheorem{theorem}{Theorem}
\newtheorem{lemma}{Lemma}

\newcommand{\dif}{\mathrm{d}}
\newcommand{\dtau}{\dif \tau}
\newcommand{\dt}{\dif t}
\newcommand{\ds}{\dif s}
\newcommand{\dx}{\dif x}
\newcommand{\dy}{\dif y}
\newcommand{\dxy}{\dif x \dif y}

\newcommand{\erf}{\mathrm{erf}}
\newcommand{\supp}{\mathrm{supp}}
\newcommand{\izo}{\int_{0}^{1}}
\newcommand{\izozo}{\izo \! \izo}
\newcommand{\linf}[1]{\left\| #1 \right\|_{\infty}}
\newcommand{\size}[2]{\left\| #1 \right\|_{#2}}
\newcommand{\ltwo}[1]{\left\| #1 \right\|_2}
\newcommand{\lu}{\left|u\right|_2}
\newcommand{\whu}{\left|U\right|_{H^{-1/4}}}

\title{An obstruction to small time local null controllability \\
for a viscous Burgers' equation}
\author{Frédéric Marbach
\footnote{Email: frederic.marbach@upmc.fr. Address: 
Laboratoire Jacques-Louis Lions, Université Pierre et Marie Curie, 
Institut Universitaire de France, 4, Place Jussieu, 75252 Paris Cedex, France.
Work supported by the ERC advanced grant 266907 (CPDENL) of the 
7th Research Framework Programme (FP7)}}

\begin{document}

\maketitle

\begin{abstract}
In this work, we are interested in the small time local null controllability for 
the viscous Burgers' equation $y_t - y_{xx} + y y_x = u(t)$ on 
the line segment $[0,1]$, with null boundary conditions. The second-hand side is
a scalar control playing a role similar to that of a pressure. In this setting, 
the classical Lie bracket necessary condition $[f_1,[f_1,f_0]]$ 
introduced by Sussmann fails to 
conclude. However, using a quadratic expansion of our system, we exhibit a 
second order obstruction to small time local null controllability. This 
obstruction holds although the 
information propagation speed is infinite for the Burgers equation. Our 
obstruction involves the weak $H^{-5/4}$ norm of the control $u$. The proof 
requires the careful derivation of an integral 
kernel operator and the estimation of residues by means of \emph{weakly singular 
integral operator} estimates.
\end{abstract}

\section{Introduction}

\subsection{Description of the system and our main result}

For $T > 0$ a small positive time, we consider the line segment 
$x \in [0,1]$ and the following one-dimensional viscous Burgers' controlled 
system:
\begin{equation} \label{system.burgers}
 \left\{ 
 \begin{aligned}
    y_t - y_{xx} + y y_x & = u(t)
    && \text{in } (0,T) \times (0,1), \\
    y(t,0) &= 0 && \text{in } (0,T), \\
    y(t,1) &= 0 && \text{in } (0,T), \\
    y(0,x) &= y_0(x) && \text{in } (0,1).
 \end{aligned}
 \right.
\end{equation}
The scalar control $u \in L^2(0,T)$ plays a role somewhat similar to that of a 
pressure for multi-dimensional fluid systems. Unlike some other studies, our 
control term $u$ depends only on time and not on the space variable. It
is supported on the whole segment $[0,1]$. For any initial data 
$y_0 \in H^1_0(0,1)$ and any fixed control $u \in L^2(0,T)$,
it can be shown (see Lemma~\ref{lemma.well-posedness} below)
that system~\eqref{system.burgers} has a unique solution in the 
space $X_T = L^2((0,T); H^2(0,1)) \cap H^1((0,T); L^2(0,1))$. We are interested in
the behavior of this system in the vicinity of the null equilibrium state.

\begin{definition} \label{def.stlc}
We say that system~\eqref{system.burgers} is \emph{small time locally null
controllable} if, for any small time $T > 0$, for any small size of the control
$\eta > 0$, there exists a region $\delta > 0$ such that:
\begin{equation}
 \forall | y_0 |_{H^1_0} \leq \delta, 
 \exists u \in L^2(0,T), 
 |u|_2 \leq \eta 
 \textrm{ such that } 
 y(T, \cdot) = 0,
\end{equation}
where $y\in X_T$ is the solution to system~\eqref{system.burgers} with initial 
condition $y_0$ and control $u$.
\end{definition}

\begin{theorem} \label{thm.frederic}
System~\eqref{system.burgers} is not \emph{small time locally null controllable}.
Indeed, there exist $T, \eta >0$ such that, for any $\delta > 0$, there exists 
$y_0 \in H^1_0(0,1)$ with $|y_0|_{H^1_0} \leq \delta$ such that, for any control 
$u \in L^2(0,T)$ with $|u|_2 \leq \eta$, the solution $y \in X_T$ 
to~\eqref{system.burgers} satisfies $y(T, \cdot) \neq 0$.
\end{theorem}

We will see in the sequel that our proof actually provides a stronger result.
Indeed, we prove that, for small times and small controls, whatever the small 
initial data $y_0$, the state $y(t)$ drifts towards  a fixed direction. Of 
course, this prevents small time local null controllability as a direct 
consequence.

\subsection{Motivation: small time obstructions despite infinite propagation speed}

As an example, let us consider the following transport control system:
\begin{equation} \label{system.transport}
 \left\{ 
 \begin{aligned}
    y_t + M y_{x} & = 0
    && \text{in } (0,T) \times (0,L), \\
    y(t,0) &= v_0(t) && \text{in } (0,T), \\
    y(0,x) &= y_0(x) && \text{in } (0,L),
 \end{aligned}
 \right.
\end{equation}
where $T > 0$ is the total time, $M > 0$ the propagation speed and $L > 0$ the 
length of the domain. The control is the boundary data $v_0$. No condition is
imposed at $x = 1$ since the characteristics flow out of the domain. For 
system~\eqref{system.transport}, small time local null controllability cannot 
hold. Indeed, even if the initial data $y_0$ is very small, the control is only
propagated towards the right at speed $M$. Thus, if $T < L/M$, controllablity
does not hold. Of course, if $T \geq L/M$, the characteristics method allows to 
construct an explicit control to reach any final state $y_1$ at time $T$.
In this context, the obstruction to controllability comes from the fact that the
information propagation speed is bounded. Indeed, let us modify slightly 
system~\eqref{system.transport} into:
\begin{equation} \label{system.transport.heat}
 \left\{ 
 \begin{aligned}
    y_t - \nu y_{xx} + M y_{x} & = 0
    && \text{in } (0,T) \times (0,L), \\
    y(t,0) &= v_0(t) && \text{in } (0,T), \\
    y(t,1) &= 0 && \text{in } (0,T), \\
    y(0,x) &= y_0(x) && \text{in } (0,L),
 \end{aligned}
 \right.
\end{equation}
where $\nu > 0$ is a (very small) viscosity. This system is small time 
globally null controllable, for any $\nu > 0$. Of course, the cost of controllablity
must explode as $\nu \rightarrow 0$ if $T$ is too small (see~\cite{MR2176274} 
for a precise study of the cost of controllability 
for~\eqref{system.transport.heat}). What we want to underline here, is that the 
infinite information propagation speed yields (at least in this context) small 
time local controllability.

\skipline

Therefore, there is a strong interest for systems where small time local
controllability does not hold \emph{despite} an infinite information propagation 
speed. 

An example of such a system is the control of a quantum particle in a 
moving potential well (box). This is a bilinear controllability problem for the 
Schrödinger equation. For such system, it can be shown that large time 
controllablity holds (see~\cite{MR2144647} if only the particle needs to be
controlled or~\cite{MR2200740} to control both the particle and the box).
For small times, negative results have been obtained by
Coron in~\cite{MR2193655} (when one tries to control both the particle and the 
position of the box), by Beauchard, Coron and Teissman in~\cite{MR3250374}
for large controls (but smooth potentials) and by Beauchard and Morancey 
in~\cite{MR3167929} (under an 
assumption corresponding to a Lie-bracket condition $[f_1,[f_1,f_0]] \neq 0$).
This last paper is related to ours since their proof relies on a coercivity
estimate involving the $H^{-1}$ norm of the control. This is natural as we will
see in paragraph~\ref{section.finite.dim}. We refer the reader to these papers
for more details and surveys on the controllability of Schrödinger equations.

Theorem~\ref{thm.frederic} can be seen as another example of a situation 
(in the context of fluid dynamics) where
small time local controllability fails despite an infinite propagation speed.

\subsection{Previous works concerning Burgers' controllability}

Let us recall known results concerning the controllability of the
viscous Burgers' equation. More generally, we introduce the following system:
\begin{equation} \label{system.burgers.full}
 \left\{ 
 \begin{aligned}
    y_t - y_{xx} + y y_x & = u(t)
    && \text{in } (0,T) \times (0,1), \\
    y(t,0) &= v_0(t) && \text{in } (0,T), \\
    y(t,1) &= v_1(t) && \text{in } (0,T), \\
    y(0,x) &= y_0(x) && \text{in } (0,1),
 \end{aligned}
 \right.
\end{equation}
where $v_0$ and $v_1$ are seen as additional controls with respect to the single
control $u$ of system~\eqref{system.burgers}. Various settings have been studied
(with either one or two boundary controls, with or without~$u$).
Once again, here $u$ only depends on $t$ and not on $x$. Some studies have been
carried out with $v_0 = v_1 = 0$ and a source term $u(t,x) \chi_{[a,b]}$ for 
$0 < a < b < 1$. However, these studies are equivalent to boundary controls 
thanks to the usual domain extension argument. Up to our knowledge, 
Theorem~\ref{thm.frederic} is the first result concerning the case without any 
boundary control and a scalar control $u$.

\skipline

\textbf{We start with results involving only a single boundary control (either
$v_0$ or $v_1$ by symmetry) and $u = 0$.}

In~\cite{MR1406566}, Fursikov and Imanuvilov prove small time 
local controllability in the vicinity of trajectories of 
system~\eqref{system.burgers.full}. Their proof relies on Carleman 
estimates for the parabolic problem obtained by seeing the non-linear term 
$yy_x$ as a small forcing term.

Global controllability towards steady states of 
system~\eqref{system.burgers.full} is possible in large time. Such studies have
been carried out by Fursikov and Imanuvilov in~\cite{MR1348646} for large
time global controllability towards all steady states, and by Coron in
\cite{MR2376661} for global null controllability in bounded time (ie. bounded 
independently on the initial data).

However, small time global controllability does not hold. The first obstruction
was obtained by Diaz in~\cite{MR1364638}. He gives a restriction for the set of 
attainable states starting from $0$. Indeed, they must lie under some limit 
state corresponding to an infinite boundary control $v_1 = + \infty$.

Fern\'andez-Cara and Guerrero derived an asymptotic of the minimal 
null-controllability time $T(r)$ for initial states of $H^1$ norm lower than $r$
(see~\cite{MR2311198}). This shows that the system is not small time globally 
null controllable.

\skipline

\textbf{We move on to two boundary controls $v_0$ and $v_1$, still with $u=0$.}
Guerrero and Imanuvilov prove in~\cite{MR2371111} that
neither small time null controllability nor bounded time global 
controllability hold in this context. Hence, controlling the whole boundary does
not provide better controllability properties.

\skipline

\textbf{When three scalar controls (namely $u(t)$, $v_0$ and $v_1$) are 
used}, Chapouly has shown in~\cite{MR2516179} that the system is small time 
globally exactly controllable to the trajectories. Her proof relies on the 
return method and on the fact that the corresponding inviscid Burgers' system 
is small time exactly controllable (see~\cite[Chapter 6]{MR2302744} for other 
examples of this method applied to Euler or Navier-Stokes).

\skipline

\textbf{When $v_1 = 0$, but $u$ and $v_0$ are controlled,} the author proved
in~\cite{MR3227326} that small time global null controllability holds. Indeed,
although a boundary layer appears near the uncontrolled part of the boundary
at $x = 1$, precise estimation of the creation and dissipation of the boundary
layer allows to conclude.

\skipline

\textbf{Concerning the controllability of the inviscid Burgers' equation}, some 
works have be carried out. In~\cite{MR1616586}, Ancona and Marson describe the 
set of attainable states in a pointwise way for the Burgers' equation on the 
half-line $x \geq 0$ with only one boundary control at $x = 0$. In~\cite{MR1612027}, 
Horsin describes the set of attainable states for a Burgers' equation on a line
segment with two boundary controls. Thorough studies are also carried out 
in~\cite{adimurthi} by Adimurthi et al. In~\cite{perrollaz}, Perrollaz studies the
controllability of the inviscid Burgers' equation in the context of entropy 
solutions with the additional control~$u(\cdot)$ and two boundary controls.

\subsection{A quadratic approximation for the non-linear system}
\label{subsection.scaling}

Starting now, we introduce $\varepsilon = T$ to remember that the total allowed
time for controllability is small. Moreover, we want to use the well-known 
scaling trading \emph{small time} with \emph{small viscosity} for viscous
fluid equations. Therefore, we introduce, for $t\in(0,1)$ and $x\in(0,1)$,
$\tilde{y}(t, x) = \varepsilon y(\varepsilon t, x)$. Hence, $\tilde{y}$ is the 
solution to:
\begin{equation} \label{system.burgers.eps}
 \left\{ 
 \begin{aligned}
    \tilde{y}_t - \varepsilon \tilde{y}_{xx} + \tilde{y} \tilde{y}_x 
    & = \tilde{u}(t) && \text{in } (0,1) \times (0,1), \\
    \tilde{y}(t,0) &= 0 && \text{in } (0,1), \\
    \tilde{y}(t,1) &= 0 && \text{in } (0,1), \\
    \tilde{y}(0,x) &= \tilde{y}_0(x) && \text{in } (0,1),
 \end{aligned}
 \right.
\end{equation}
where $\tilde{u}(t) = \varepsilon^2 u(\varepsilon t)$ and 
$\tilde{y}_0 = \varepsilon y_0$. This scaling is widely used for controllability
results since small viscosity developments are easier to handle. As we will
prove in section~\ref{section.back}, system~\eqref{system.burgers.eps} 
can help us to deduce results for system~\eqref{system.burgers}. To further 
simplify the computations in the following sections, let us drop the tilda signs 
and the initial data. Therefore, we will study the behavior of the following 
system near $y \equiv 0$:
\begin{equation} \label{system.y}
 \left\{ 
 \begin{aligned}
    y_t - \varepsilon y_{xx} + y y_x & = u(t)
    && \text{in } (0,1) \times (0,1), \\
    y(t,0) &= 0 && \text{in } (0,1), \\
    y(t,1) &= 0 && \text{in } (0,1), \\
    y(0,x) &= 0 && \text{in } (0,1).
 \end{aligned}
 \right.
\end{equation}
Properties proven on system~\eqref{system.y} will easily be translated into
properties for system~\eqref{system.burgers} in Section~\ref{section.back}. 
Moreover, since we are studying local null controllability, both the control 
$u$ and the state $y$ are small. Thus, if $\eta$ describes the size of the 
control as in Definition~\ref{def.stlc}, let us name our control $\eta u(t)$,
with $u$ of size $\mathcal{O}(1)$. We expand $y$ as 
$y = \eta a + \eta^2 b + \mathcal{O}(\eta^3)$, and we compute the associated 
systems:
\begin{equation} \label{system.a}
 \left\{ 
 \begin{aligned}
    a_t - \varepsilon a_{xx}  
    & = u(t) && \text{in } (0,1) \times (0,1), \\
    a(t,0) &= 0 && \text{in } (0,1), \\
    a(t,1) &= 0 && \text{in } (0,1), \\
    a(0,x) &= 0 && \text{in } (0,1)
 \end{aligned}
 \right.
\end{equation}
and
\begin{equation} \label{system.b}
 \left\{ 
 \begin{aligned}
    b_t - \varepsilon b_{xx}  
    & = -aa_x && \text{in } (0,1) \times (0,1), \\
    b(t,0) &= 0 && \text{in } (0,1), \\
    b(t,1) &= 0 && \text{in } (0,1), \\
    b(0,x) &= 0 && \text{in } (0,1).
 \end{aligned}
 \right.
\end{equation}
It is easy to see that system~\eqref{system.a} is not controllable. 
Indeed, the control $u(t)$ can actually be written as 
$u(t)\chi_{[0,1]}$, and $\chi_{[0,1]}$ is an even function
on the line segment $[0,1]$. Thus, the control only acts on even modes of $a$.
In the linearized system~\eqref{system.a}, all odd modes evolve freely.
This motivates the second order expansion of our Burgers' system in order to 
understand its controllability properties using $b$. Given 
systems~\eqref{system.a} and~\eqref{system.b}, we know that $a$ is even and 
$b$ is odd. 

\subsection{A finite dimensional counterpart}
\label{section.finite.dim}

Systems~\eqref{system.a} and~\eqref{system.b} exhibit an interesting structure.
Indeed, the first system is fully controllable (if we consider that $a$ lives
within the subspace of even functions), while the second system is indirectly
controlled through a quadratic form depending on $a$. Let us introduce the 
following finite dimensional control system:
\begin{equation} \label{system.ab}
 \left\{ 
 \begin{aligned}
 	\dot{a} & = M a + u(t) m && \text{in } (0,T), \\
 	\dot{b} & = L b + Q(a,a) && \text{in } (0,T),
 \end{aligned}
 \right.
\end{equation}
where the states $a(t), b(t) \in \R^n \times \R^p$, $M$ is an $n\times n$ 
matrix, $m$ is a fixed vector in $\R^n$ along which the scalar control acts,
$L$ is a $p \times p$ matrix and $Q$ is a quadratic function from 
$\R^n\times\R^n$ into $\R^p$. Moreover, we assume that the pair $(M,m)$
satisfies the classical Kalman rank condition 
(see~\cite[Theorem 1.16]{MR2302744}). Hence, the state $a$ is fully
controllable. We consider the small time null controllability problem 
for system~\eqref{system.ab}. We want to know, if, for any $T > 0$, for any 
initial state $(a^0,b^0)$, there exists a control $u : (0,T) \rightarrow \R$ 
such that the solution to~\eqref{system.ab} satisfies $a(T) = 0$ and $b(T) = 0$. 
As proved in~\cite{MR3110058} for the case $L = 0$, the answer to this question 
is always no in finite dimension, whatever $M, m, L$ and $Q$. 

System~\eqref{system.ab} is a particular case of the more general class of 
control affine systems. Indeed, if we let $x(t) = (a(t),b(t)) \in \R^{n+p}$,
we can write system~\eqref{system.ab} as:
\begin{equation} \label{eq.control.affine}
  \dot{x} = f_0(x) + u(t) f_1(x),
\end{equation}
where $f_0(x) = (Ma, Lb + Q(a,a))$ and $f_1(x) = (m, 0)$. The controllability
of systems like~\eqref{eq.control.affine} is deeply linked to the iterated Lie 
brackets of the vector fields $f_0$ and $f_1$ (see~\cite[Section 3.2]{MR2302744} 
for a review). 

\skipline

Let us give a few examples with $n = 3$. We write $a = (a_1,a_2,a_3)$ and
we consider the system:
\begin{equation} \label{a.123}
  \dot{a}_1 = a_2, \quad
  \dot{a}_2 = a_3, \quad
  \dot{a}_3 = u.
\end{equation}
Although the strong structure of equation~\eqref{a.123} can seem a little 
artificial, it is in fact the general case. Indeed, up to a translation of the
control, controllable systems can always be brought back to this canonical 
form introduced by Brunovsky in~\cite{MR0284247} (for a proof, 
see~\cite[Theorem 2.2.7]{MR2224013}). The resulting system is \emph{flat}. We
can express the full state as derivatives of a single scalar function. Indeed, 
if we let $\theta = a_1$, we have $a_2 = \theta'$, $a_3 = \theta''$ and 
$u = \theta'''$. If we choose an initial state $(a^0, b^0)$ with
$a^0 = 0$, we obtain $\theta(0) = \theta'(0) = \theta''(0) = 0$. Moreover, if 
we assume that the control $u$ drives the state $(a,b)$ to $(0,0)$ at time $T$,
we also have  $\theta(T) = \theta'(T) = \theta''(T) = 0$. These conditions allow
integration by parts without boundary terms.

\skipline

To keep the examples simple, we choose $p = 1$ (hence $b = b_1 \in \R$) and we 
let $L = 0$. 

\textbf{First example.} We consider the evolution $\dot{b} = a_2^2 + a_1 a_3$.
If the initial state is $(a^0,b^0)$ where $a^0 = 0$, we can compute 
$b(T) = b^0 + \int_0^T \theta'^2(t) + \theta(t)\theta''(t)\dt = b^0$. Hence, 
null controllability does not hold since any control driving $a$ from $0$
back to $0$ has no action on $b$. This obstruction to controllability is linked
to the fact that $\textrm{dim} \mathcal{L}(0) = 3$, where $\mathcal{L}$ is the
Lie algebra generated by $f_0$ and $f_1$. The system is locally constrained to 
evolve within a $3$ dimensional manifold of $\R^4$. Indeed, the evolution
equation can be rephrased as $\dot{b} = \frac{\dif}{\dt} (a_1a_2)$. Thus, the
quantity $b - a_1 a_2$ is a constant (\emph{conservation law} of the system).

\textbf{Second example.} We consider the evolution $\dot{b} = a_3^2$. Thus,
$b(T) = b^0 + \int_0^T \theta''(t)^2 \dt$. This is also an obstruction to null
controllability. Indeed, all choices of control will make $b$ increase.
In this setting, we recover the well known second order Lie bracket condition
discovered by Sussmann (see~\cite[Proposition 6.3]{MR710995}). Indeed, here,
$[f_1, [f_1, f_0]] = (0_{\R^3}, Q(m,m)) = (0_{\R^3}, 1)$. 
System~\eqref{eq.control.affine} drifts in the direction $[f_1,[f_1,f_0]]$ 
and the control cannot prevent it because this direction does not belong
to the set of the first order controllable directions $(m,0), (Mm,0)$ and 
$(M^2m,0)$ (Lie brackets of $f_0$ and $f_1$ involving $f_1$ once and only once).

\textbf{Third example.} We consider $\dot{b} = a_2^2$. Thus,
$b(T) = b^0 + \int_0^T \theta'^2(t) \dt$. Again, $b$ can only increase.
Here, the first \emph{bad} Lie bracket $[f_1,[f_1,f_0]]$ vanishes for $x = 0$.
However, we can check that $[f_1,[f_0,[f_0,[f_1,f_0]]]] = (0_{\R^3}, Q(Mm, Mm)) 
= (0_{\R^3}, 1)$. Compared with the second example, the increase of $b$ is
weaker. Indeed, in the second example, we had $b(T) = b^0 + |u|_{H^{-1}(0,T)}^2$.
In this third example, $b(T) = b^0 + |u|_{H^{-2}(0,T)}^2$.

\skipline

Although these examples may seem caricatural, they reflect the general case.
In finite dimension, systems like~\eqref{system.ab} are never small time 
controllable. Either because they evolve within a strict manifold, or because
some quantity depending on $b$ increases. Moreover, the amount by which $b$
increases is linked to the order of the first bad Lie bracket and can be
expressed as a weak norm depending on the control. One of the goals of our
work is thus also to investigate the situation in infinite dimension, where
Lie brackets are harder to define and compute.

Therefore, the first natural question is to compute the Lie bracket
$[f_1,[f_1,f_0]](0)$ for systems~\eqref{system.a} and~\eqref{system.b}. As we
have seen in finite dimension, this Lie bracket is $(0, Q(m,m))$. In our 
setting, $m$ is the even function $\chi_{[0,1]}$ and 
$Q(a,a) = -aa_x$. Thus $Q(m,m)$ is null. This can be proved computationally 
using Fourier series expansions. Let us give a much simpler argument inspired by
the formal fact that $\partial_x 1 = 0$. For any $a \in L^2(0,1)$ and any smooth 
test function $\phi$ such that $\phi(0) = \phi(1) = 0$, we have:
\begin{equation} \label{lie.q.11}
\int_0^1 Q(a,a) \phi = \frac{1}{2} \int_0^1 a^2(x) \phi_x(x) \dx.
\end{equation}
Hence, even if $q := Q(1,1)$ was defined in a very weak sense,~\eqref{lie.q.11}
yields:
\begin{equation} \label{lie.q.11.2}
 \langle q, \phi \rangle 
 = \frac{1}{2}\int_0^1 \phi_x 
 = \frac{1}{2}\phi(1) - \frac{1}{2}\phi(0) 
 = 0
\end{equation}
Since~\eqref{lie.q.11.2} is valid for any smooth $\phi$ null at the boundaries,
we conclude that indeed, $q = Q(1,1)$ is null. Therefore, the classical 
$[f_1,[f_1,f_0]]$ necessary condition by Sussmann does not provide
an obstruction to small time controllability for our system. This also explains 
why the coercivity property we are going to prove is in a weaker norm than $H^{-1}$.

\subsection{Strategy for the proof}

Most of this paper is dedicated to the asymptotic study of 
systems~\eqref{system.a} and~\eqref{system.b} as the viscosity $\varepsilon$ 
tends to zero. In Section~\ref{section.back}, we prove that this study is 
sufficient to conclude about the local null controllability for 
system~\eqref{system.burgers}. In order to prove that 
system~\eqref{system.burgers} is not small time locally
null controllable, we intend to exhibit a quantity depending on the state 
$y(t,\cdot)$ that cannot be controlled. For $\rho \in H^1(0,1)$, we will 
consider quantities of the form $\langle \rho , y(t, \cdot) \rangle$.

Looking at system~\eqref{system.b} when $\varepsilon$ is very small, we get the 
idea to consider $\rho(x) = x - \frac{1}{2}$. Indeed, we obtain:
\begin{equation} \label{b.rho.x}
 \frac{\dif}{\dt} \int_0^1 \rho(x) b(t,x) \dx
 = \frac{1}{2} \int_0^1 a^2(t,x) \dx
 + \frac{\varepsilon}{2} \left( b_x(t,1) - b_x(t,0) \right).
\end{equation}
Formally, if we let $\varepsilon = 0$ in equation~\eqref{b.rho.x}, it is very
encouraging because it shows that the quantity $\langle \rho, b \rangle$ can 
only increase, whatever is the choice of the control. Moreover, since we can 
compute the amount by which it increases, we have a kind of coercivity and we 
can hope to be able to use it to overwhelm both residues coming from the fact 
that $\varepsilon > 0$ and residues between the quadratic approximation and 
the full non-linear system. Sadly, the second term in the right-hand side of
equation~\eqref{b.rho.x} is hard to handle. However, as $a$ depends linearly
on $u$, and $b$ depends quadratically on $a$, we expect that we can find a 
kernel $K^\varepsilon(s_1,s_2)$ such that:
\begin{equation} \label{b.1.rho.k}
 \langle \rho, b(1, \cdot) \rangle
 = \int_0^1 \int_0^1 K^\varepsilon(s_1,s_2) u(s_1) u(s_2) \ds_1\ds_2.
\end{equation}
Thanks to equation~\eqref{b.rho.x}, we expect that~\eqref{b.1.rho.k} actually
defines a positive definite kernel acting on $u$, allowing us to use its 
coercivity to overwhelm various residues.

\skipline

In Section~\ref{section.technical}, we recall a set of technical well-posedness 
estimates for heat and Burgers systems.

In Section~\ref{section.burgers.k0}, we show that formula~\eqref{b.1.rho.k} 
holds and we give an explicit construction of the kernel~$K^\varepsilon$. 
Moreover, we compute formally its limit $K^0$ as $\varepsilon \rightarrow 0$.

In Section~\ref{section.k0.coercive}, we prove that the kernel $K^0$ is coercive 
with respect to the $H^{-5/4}(0,1)$ norm of the control~$u$, by recognizing a 
Riesz potential and a fractional laplacian.

In Section~\ref{section.keps.residues}, we use weakly singular integral operator
estimates to bound the residues between $K^\varepsilon$ and $K^0$ and thus 
deduce that $K^\varepsilon$ is also coercive, for $\varepsilon$ small enough.

In Section~\ref{section.back}, we use these results to go back to the 
controllability of Burgers.

In Appendix~\ref{appendix.wsio}, we give a short presentation of the theory of
weakly singular integral operators and a sketch of proof of the main estimation
lemma we use.

\section{Preliminary technical lemmas}
\label{section.technical}

In this section, we recall a few useful lemmas and estimates, mostly concerning
the heat equation and Burgers equation on a line segment. Throughout this 
section, $\nu$ is a positive viscosity and $T$ a positive time. To lighten the
computations, we will use the notation $\lesssim$ to denote inequalities that 
hold up to a numerical constant. We will not attempt to keep track of these 
numerical constants. We insist on the fact that these constants do not depend 
on any parameter (neither the time $T$, nor the viscosity $\nu$, the control 
$u$, or any other unknown).

\subsection{Properties of the space $X_T$}

We recall the definition given in the introduction and state without proof the 
following classical lemmas which can be proved using either interpolation theory
or Fourier transforms with respect to time and space.

\begin{definition} \label{def.xt}
We define the functional space:
\begin{equation} \label{def.X}
X_T = L^2\left((0,T), H^2(0,1)\right) \cap H^1\left((0,T),L^2(0,1)\right).
\end{equation}
We endow the space $X_T$ with the scaling invariant norm:
\begin{equation} \label{def.norm.xt}
  \left\| z \right\|_{X_T} := 
    T^{-1/2} \left\|z\right\|_2 
    + T^{-1/2} \left\|z_{xx}\right\|_2
    + T^{1/2} \left\|z_{t}\right\|_2.
\end{equation}
\end{definition}

\begin{lemma} \label{lemma.xt.c0}
$X_T \hookrightarrow \mathcal{C}^0([0,T], H^1(0,1))$.
Moreover, for any function $z \in X_T$,
\begin{equation} \label{eq.lemma.xt.c0.1}
  \sup_{t \in [0,T]} | z(t, \cdot) |_{H^1(0,1)}
  \lesssim
  \left\| z \right\|_{X_T}.
\end{equation}
In particular,
\begin{equation} \label{eq.lemma.xt.c0.2}
  \left\| z \right\|_\infty
  \lesssim 
  \left\| z \right\|_{X_T}.
\end{equation}
\end{lemma}

\begin{lemma} \label{lemma.xt.zx}
For any $z \in X_T$, the boundary traces of $z_x$ satisfy:
\begin{equation} \label{eq.lemma.xt.zx.2}
  T^{-1/4} \left| z_x(\cdot, 0) \right|_{H^{1/4}(0,T)}
  + T^{-1/4} \left| z_x(\cdot, 1) \right|_{H^{1/4}(0,T)}
  \lesssim \left\| z \right\|_{X_T}.
\end{equation}
\end{lemma}

\subsection{Smooth setting for the heat equation}

We start by recalling standard estimates in a smooth (strong) setting for one 
dimensional heat equations that will be useful in the sequel. We state all 
results for standard forward heat equations, but the same results hold for 
backwards heat equations with final time conditions.

\begin{lemma} \label{lemma.heat.f}
Let $f \in L^2((0,T)\times(0,1))$ and $z^0 \in H^1_0(0,1)$. We consider
the system:
\begin{equation} \label{system.heat.f}
 \left\{ 
 \begin{aligned}
    z_t - \nu z_{xx} & = f && \text{in } (0,T) \times (0,1), \\
    z(t,0) &= 0 && \text{in } (0,T), \\
    z(t,1) &= 0 && \text{in } (0,T), \\
    z(0,x) &= z^0(x) && \text{in } (0,1).
 \end{aligned}
 \right.
\end{equation}
There is a unique solution $z \in X_T$ to system~\eqref{system.heat.f}. 
Moreover, it satisfies the estimate:
\begin{equation} \label{estimate.lemma.heat.f}
 \nu \left\| z_{xx} \right\|_2
 + \sqrt{\nu} \left\| z_{x} \right\|_2
 + \left\| z_{t} \right\|_2
 \lesssim \left\| f \right\|_2 + \sqrt{\nu} | z^0_x |_2. 
\end{equation}
\end{lemma}
\begin{proof}
The proof of the existence and uniqueness is standard. Let us recall how we can
obtain estimate~\eqref{estimate.lemma.heat.f}. We multiply 
equation~\eqref{system.heat.f} by $z_{xx}$ and integrate by parts over 
$x \in (0,1)$. Thus,
\begin{equation} \label{eq.lemma.heat.f.1}
  \frac{\dif}{\dt} \left[ \frac{1}{2} \int_0^1 z_x^2 \right]
  + \nu \int_0^1 z_{xx}^2 = - \int_0^1 f z_{xx}.
\end{equation}
For any $T' < T$, we can integrate~\eqref{eq.lemma.heat.f.1} over 
$t \in (0,T')$. Hence, we obtain:
\begin{equation} \label{eq.lemma.heat.f.2}
  \frac{1}{2} | z_x(T') |_2^2 + \nu \int_0^{T'} \int_0^1 z_{xx}^2 
  = - \int_0^{T'}\int_0^1 f z_{xx} + \frac{1}{2} \left|z^0_x\right|_2^2.
\end{equation}
From~\eqref{eq.lemma.heat.f.2}, we easily deduce that:
\begin{align}
  \nu \left\| z_{xx} \right\|_2 
  & \lesssim \left\| f \right\|_{L^2} + \sqrt{\nu} | z^0_x |_2,
  \label{eq.lemma.heat.f.3} \\
  \sqrt{\nu} \left\| z_{x} \right\|_{L^\infty(L^2)} 
  & \lesssim \left\| f \right\|_{L^2} + \sqrt{\nu} | z^0_x |_2
  \label{eq.lemma.heat.f.4}.
\end{align}
Eventually, we obtain estimate~\eqref{estimate.lemma.heat.f}
from estimates~\eqref{eq.lemma.heat.f.3} and~\eqref{eq.lemma.heat.f.4}
since we can write $z_t$ as $f + \nu z_{xx}$.
\end{proof}

\begin{lemma} \label{lemma.max}
Let $z^0 \in H^1_0(0,1)$ and consider $z \in X_T$ the solution to 
system~\eqref{system.heat.f} with a null forcing term ($f = 0$). It satisfies:
\begin{equation} \label{eq.lemma.max}
 \left\| z \right\|_\infty \leq \left| z^0 \right|_\infty.
\end{equation}
\end{lemma}

\begin{proof}
Although~\eqref{eq.lemma.max} is not a direct consequence of the combination
of~\eqref{eq.lemma.xt.c0.2} and~\eqref{estimate.lemma.heat.f} (which would yield
a weaker conclusion), it can be obtained via a standard application of the 
maximum principle, which can be applied in this strong setting.
\end{proof}

\subsection{Weaker settings for the heat equation}
\label{section.heat.weak}

Let us move on to weaker settings for the heat equation. Moreover, we introduce
inhomogeneous boundary data as we will need them in the sequel.

\begin{definition} \label{definition.weak}
Let $f \in (X_T)'$, $v_0,v_1 \in H^{-1/4}(0,T)$ and 
$z^0 \in H^{-1}(0,1)$. We consider:
\begin{equation} \label{system.heat.weak}
 \left\{ 
 \begin{aligned}
    z_t - \nu z_{xx} & = f && \text{in } (0,T) \times (0,1), \\
    z(t,0) &= v_0(t) && \text{in } (0,T), \\
    z(t,1) &= v_1(t) && \text{in } (0,T), \\
    z(0,x) &= z^0(x) && \text{in } (0,1).
 \end{aligned}
 \right.
\end{equation}
We say that $z \in L^2((0,T)\times(0,1))$ is a weak solution to 
system~\eqref{system.heat.weak} if, for all $g \in L^2((0,T)\times(0,1))$,
\begin{equation} \label{heat.transposition}
 \begin{split}
  \langle z, g \rangle_{L^2,L^2} = 
  \langle f, \varphi \rangle_{(X_T)',X_T}
  &  + \langle z^0, \varphi(0, \cdot) \rangle_{H^{-1}(0,1),H^1_0(0,1)} \\
  & + \nu \langle v_0, \varphi_x(\cdot, 0) \rangle_{H^{-1/4}(0,T),H^{1/4}(0,T)} \\
  &  - \nu \langle v_1, \varphi_x(\cdot, 1) \rangle_{H^{-1/4}(0,T),H^{1/4}(0,T)},
 \end{split}
\end{equation}
where $\varphi \in X_T$ is the solution to the dual system:
\begin{equation} \label{system.heat.dual.phi}
 \left\{ 
 \begin{aligned}
    \varphi_t + \nu \varphi_{xx} & = - g && \text{in } (0,T) \times (0,1), \\
    \varphi(t,0) &= 0 && \text{in } (0,T), \\
    \varphi(t,1) &= 0 && \text{in } (0,T), \\
    \varphi(T,x) &= 0 && \text{in } (0,1).
 \end{aligned}
 \right.
\end{equation}
\end{definition}

\begin{lemma} \label{lemma.heat.weak}
There exists a unique weak solution $z \in L^2((0,T)\times(0,1))$ to 
system~\eqref{system.heat.weak}. Moreover:
\begin{equation} \label{estimate.weak.heat}
 \left\| z \right\|_2 
 \lesssim 
 T^{-1/2}\nu^{-1} \left( \left\| f \right\|_{(X_T)'} + | z^0 |_{H^{-1}} \right)
 + T^{-1/4} \left( |v_0|_{H^{-1/4}} + |v_1|_{H^{-1/4}} \right).
\end{equation}
\end{lemma}
\begin{proof}
For any $g \in L^2((0,T)\times(0,1))$, Lemma~\ref{lemma.heat.f} asserts that 
system~\eqref{system.heat.dual.phi} admits a unique solution $\varphi \in~X_T$ 
such that $\|\varphi\|_{X_T} \lesssim T^{-1/2}\nu^{-1} \| g \|_{L^2}$. 
Moreover, thanks to estimates~\eqref{eq.lemma.xt.c0.1}
and~\eqref{eq.lemma.xt.zx.2}, the right-hand
side of equation~\eqref{heat.transposition} defines a continuous linear form
on $L^2$. The Riesz representation theorem therefore proves the existence of a 
unique $z \in L^2$ satisfying estimate~\eqref{estimate.weak.heat}.
\end{proof}

\begin{lemma} \label{lemma.heat.fx}
Let $f \in L^2((0,T)\times(0,1))$. We consider the following heat system:
\begin{equation} \label{system.heat.fx}
 \left\{ 
 \begin{aligned}
    z_t - \nu z_{xx} & = f_x && \text{in } (0,1) \times (0,1), \\
    z(t,0) &= 0 && \text{in } (0,1), \\
    z(t,1) &= 0 && \text{in } (0,1), \\
    z(0,x) &= 0 && \text{in } (0,1).
 \end{aligned}
 \right.
\end{equation}
There is a unique solution $z \in L^2((0,T)\times(0,1))$ to 
system~\eqref{system.heat.fx}. Moreover, it satisfies the estimate:
\begin{equation} \label{estimate.lemma.heat.fx}
 \nu^{1/2} \left\| z \right\|_{L^\infty(L^2)} 
 + \nu \left\| z_x \right\|_{L^2} 
 \lesssim
 \left\| f \right\|_{L^2}.
\end{equation}
\end{lemma}

\begin{proof}
For $f \in L^2$, it is easy to check that $f_x \in X_T'$. Hence, we can apply 
Lemma~\ref{lemma.heat.weak} and system~\eqref{system.heat.fx} has a unique 
solution $z \in L^2$. In fact, this solution is even smoother. 
Estimate~\eqref{estimate.lemma.heat.fx} is obtained as usual by multiplying
equation~\eqref{system.heat.fx} by $z$ and integration by parts.
\end{proof}

\subsection{Burgers and forced Burgers systems}

We move on to Burgers-like systems. For the sake of completeness, we provide
a short proof of the existence of a solution to system~\eqref{system.burgers} 
and a precise estimate for forced Burgers-like systems that will be necessary
in the sequel.

\begin{lemma} \label{lemma.burgers.forced}
Let $w \in X_T$, $g \in L^2((0,T),H^1(0,1))$ and 
$y^0 \in H^1_0(0,1)$. We consider $y \in X_T$ a solution to the following 
forced Burgers-like system:
\begin{equation} \label{system.burgers.forced}
 \left\{ 
 \begin{aligned}
    y_t - \nu y_{xx} & = - yy_x + (wy)_x + g_x && \text{in } (0,T) \times (0,1), \\
    y(t,0) &= 0 && \text{in } (0,T), \\
    y(t,1) &= 0 && \text{in } (0,T), \\
    y(0,x) &= y^0(x) && \text{in } (0,1).
 \end{aligned}
 \right.
\end{equation}
Then,
\begin{equation} \label{estimate.burgers.forced}
 \begin{split}
 \nu \left\| y_{xx} \right\|_2
 + \sqrt{\nu} \left\| y_{x} \right\|_2
 + \left\| y_{t} \right\|_2
 \lesssim &
 \left\| g_x \right\|_2
 + e^{\gamma} \left\| w_x \right\|_{L^2(L^\infty)} 
   \left( \nu^{-1/2} \left\| g\right\|_2 + \left|y^0\right|_2^2 \right)  \\
 & + \left(1 + \sqrt{\gamma} e^{\gamma}\right) \left\| w \right\|_\infty 
 \left( \nu^{-1} \left\| g \right\|_2 + \nu^{-1/2} \left|y^0\right|_2^2 \right) \\
 & + \left(1+\sqrt{\gamma}e^{6\gamma}\right) e^\gamma \left\|g\right\|_{L^2(L^\infty)}
 \left( \nu^{-3/2} \left\|g\right\|_2 + \nu^{-1}\left|y^0\right|_2 \right) \\
 & + \left(1+\sqrt{\gamma}e^{6\gamma}\right) \nu^{-1/2} \left|y^0\right|_4^2
 + \nu^{1/2} \left|y^0_x\right|_2.
 \end{split}
\end{equation}
where we introduce $\gamma = \frac{1}{\nu} \left\| w \right\|_{L^2(L^\infty)}^2$.
\end{lemma}

\begin{proof} 
\textbf{$L^2$ estimates for $y$ and $y_x$.}
We start by multiplying equation~\eqref{system.burgers.forced} by $y$, 
and integrate by parts over $(0,1)$:
\begin{equation} \label{eq.lemma.forced.1}
 \begin{split}
  \frac{1}{2} \frac{\dif}{\dt} \int_0^1 y^2 
  + \nu \int_0^1 y_x^2
  & = -\int_0^1 w y y_x - \int_0^1 g y_x \\
  & \leq \frac{2}{2\nu} \int_0^1 w^2y^2
   + \frac{\nu}{4} \int_0^1 y_x^2
   + \frac{2}{2\nu} \int_0^1 g^2
   + \frac{\nu}{4} \int_0^1 y_x^2.
 \end{split}
\end{equation}
From~\eqref{eq.lemma.forced.1}, we deduce:
\begin{equation} \label{eq.lemma.forced.2}
 \frac{\dif}{\dt} \int_0^1 y^2 + \nu \int_0^1 y_x^2
 \leq
 \frac{2}{\nu} |w(t, \cdot)|_\infty^2 \int_0^1 y^2 
 + \frac{2}{\nu} \int_0^1 g^2. 
\end{equation}
We apply Grönwall's lemma to~\eqref{eq.lemma.forced.2} to obtain:
\begin{equation} \label{eq.lemma.forced.3}
 \left\| y \right\|_{L^\infty(L^2)}^2
 \leq
 e^{2\gamma} \left( \frac{2}{\nu} \left\| g \right\|_2^2
 + \left| y^0 \right|_2^2 \right).
\end{equation}
Plugging~\eqref{eq.lemma.forced.3} into~\eqref{eq.lemma.forced.2} yields:
\begin{equation} \label{eq.lemma.forced.4}
 \nu \left\| y_x \right\|_2^2
 \leq
 \left( 1 + 2\gamma e^{2\gamma} \right)
 \left( \frac{2}{\nu} \left\| g \right\|_2^2
   + \left| y^0 \right|_2^2 \right).
\end{equation}

\textbf{$L^2$ estimate for $yy_x$.} We repeat a similar technique, multiplying
this time equation~\eqref{system.burgers.forced} by $y^3$. Using the same 
approach yields:
\begin{equation} \label{eq.lemma.forced.5}
 \frac{\dif}{\dt} \int_0^1 y^4 + 6\nu\int_0^1 y^2y_x^2
 \leq
 \frac{12}{\nu}|w(t, \cdot)|_\infty^2 \int_0^1 y^4
 + \frac{12}{\nu} |g(t, \cdot)|_\infty^2 \int_0^1 y^2.
\end{equation}
We apply Grönwall's lemma to~\eqref{eq.lemma.forced.5} to obtain:
\begin{equation} \label{eq.lemma.forced.6}
 \left\| y \right\|_{L^\infty(L^4)}^4
 \leq
 e^{12\gamma} \left( \frac{12}{\nu} 
 \left\| g \right\|_{L^2(L^\infty)}^2 \left\| y \right\|_{L^\infty(L^2)}^2
 + \left| y^0 \right|_4^4 \right).
\end{equation}
Once again, plugging back estimate~\eqref{eq.lemma.forced.6} 
into~\eqref{eq.lemma.forced.5} gives:
\begin{equation} \label{eq.lemma.forced.7}
 6 \nu \left\| yy_x \right\|_2^2
 \leq
 \left( 1 + 12 \gamma e^{12\gamma} \right)
 \left( \frac{12}{\nu} 
 \left\| g \right\|_{L^2(L^\infty)}^2 \left\| y \right\|_{L^\infty(L^2)}^2
 + \left| y^0 \right|_4^4 \right).
\end{equation}
\textbf{Conclusion.} To conclude the proof, we use Lemma~\ref{lemma.heat.f},
with a source term $f = g_x + w_x y + w y_x - yy_x$. 
Estimate~\eqref{estimate.burgers.forced}
comes from the combination of~\eqref{estimate.lemma.heat.f} with 
equations~\eqref{eq.lemma.forced.3},~\eqref{eq.lemma.forced.4} 
and~\eqref{eq.lemma.forced.7}.
\end{proof}

\begin{lemma} \label{lemma.well-posedness}
 For any initial data $y_0 \in H^1_0(0,1)$ and any control $u \in L^2(0,T)$, 
 system~\eqref{system.burgers} has a unique solution $y \in X_T$. Moreover:
 \begin{align} 
  \label{estimate.burgers} 
  \ltwo{y_{xx}} + \ltwo{y_t}
  & \lesssim |u|_2 + |u|_2^2 + |y^0|_4^2 + |y^0_x|_2, \\
  \label{eq.maximum.burgers} \linf{y} 
  & \leq |y^0|_\infty + |u|_{L^1}.
 \end{align}
\end{lemma}

\begin{proof} 
 This type of existence result relies on standard \textit{a priori} estimates 
 and the use of a fixed point theorem. Such techniques are described 
 in~\cite{MR0259693}. One can also use a semi-group method as in~\cite{MR710486}.
 The quantitative estimate is obtained by applying 
 Lemma~\ref{lemma.burgers.forced} with $w = 0$ (hence $\gamma = 0$) 
 and $g(t,x) = xu(t)$. Equation~\eqref{estimate.burgers.forced} 
 yields ~\eqref{estimate.burgers}. The second 
 estimate~\eqref{eq.maximum.burgers} is a consequence of the maximum principle,
 which can be applied in this strong setting.
\end{proof}

\section{From Burgers to a kernel integral operator}
\label{section.burgers.k0}

\subsection{A general method for evaluating a projection}

As we mentionned in the introduction, we are going to consider a projection
of the state $b$ against some given profile $\rho(x)$ at the final time $t = 1$. 
Since $a$ depends linearly on $u$ and $b$ depends quadratically on $a$, it is 
natural to look for this projection as a quadratic integral operator acting on 
our control $u$. Indeed, let us prove the following result.

\begin{lemma} \label{lemma.k}
Let $\rho \in L^2(0,1)$ and $\varepsilon > 0$. There exists a symmetric kernel 
$K^\varepsilon \in L^\infty((0,1)^2)$ such that, for any $u \in L^2(0,1)$, the 
solution to system~\eqref{system.a}-\eqref{system.b} satisfies:
\begin{equation} \label{eq.lemma.k}
 \int_0^1 b(1,x) \rho(x) \dif x = 
 \iint_{(0,1)^2} K^\varepsilon(s_1,s_2) u(s_1) u(s_2) \ds_1\ds_2.
\end{equation}
\end{lemma}

The key point of the proof is to convert the pointwise in time projection of
$b$ into an integrated projection over the time interval $(0,1)$. Indeed, we
start with the proof of the following lemma.

\begin{lemma} \label{lemma.z.rho}
Let $f \in L^2((0,1)^2)$, $\varepsilon > 0$ and $z \in X_1$ be the solution to:
\begin{equation} \label{system.z.rho}
 \left\{ 
 \begin{aligned}
    z_t - \varepsilon z_{xx} & = f && \text{in } (0,1) \times (0,1), \\
    z(t,0) &= 0 && \text{in } (0,1), \\
    z(t,1) &= 0 && \text{in } (0,1), \\
    z(0,x) &= 0 && \text{in } (0,1).
 \end{aligned}
 \right.
\end{equation}
Take $\rho \in L^2(0,1)$. The final time projection of $z$ against $\rho$ 
satisfies:
\begin{equation} \label{eq.lemma.z.rho}
  \int_0^1 z(1,x) \rho(x) \dx = \iint_{(0,1)^2} \Phi(1-t,x) f(t,x) \dx \dt,
\end{equation}
where $\Phi \in X_1$ is the solution to:
\begin{equation} \label{def.Phi}
 \left\{ 
 \begin{aligned}
    \Phi_t - \varepsilon \Phi_{xx} & = 0 && \text{in } (0,1) \times (0,1), \\
    \Phi(t,0) &= 0 && \text{in } (0,1), \\
    \Phi(t,1) &= 0 && \text{in } (0,1), \\
    \Phi(0,x) &= \rho(x) && \text{in } (0,1).
 \end{aligned}
 \right.
\end{equation}
\end{lemma}

\begin{proof}
Let us introduce $\Psi \in X_1$, the solution to:
\begin{equation} \label{system.Psi}
 \left\{ 
 \begin{aligned}
    \Psi_t - \varepsilon \Psi_{xx} & = \rho && \text{in } (0,1) \times (0,1), \\
    \Psi(t,0) &= 0 && \text{in } (0,1), \\
    \Psi(t,1) &= 0 && \text{in } (0,1), \\
    \Psi(0,x) &= 0 && \text{in } (0,1).
 \end{aligned}
 \right.
\end{equation}
Using this system, we can convert the time punctual projection of the state $z$ 
against $\rho$ into a projection of the source term $f$ onto the full square:
\begin{equation} \label{eq.lemma.z.rho.1}
 \begin{split}
  \int_0^1 z(1,x) \rho(x) \dx 
  & = \left. \frac{\dif}{\dif T} \int_0^T \int_0^1 z(t,x) \cdot \rho(x) \dx \dt \right|_{T=1} \\
  & = \left. \frac{\dif}{\dif T} \int_0^T \int_0^1  z(t,x) 
    \cdot \{\Psi_t - \varepsilon \Psi_{xx}\}(T-t,x) \dx \dt \right|_{T=1} \\
  & = \left. \frac{\dif}{\dif T} \int_0^T \int_0^1 \{ z_t - \varepsilon z_{xx} \}(t,x) 
    \cdot \Psi(T-t,x) \dx \dt \right|_{T=1} \\
  & = \left. \frac{\dif}{\dif T} \int_0^T \int_0^1 f(t,x) 
    \cdot \Psi(T-t,x) \dx \dt \right|_{T=1} \\
  & = \int_0^1 \int_0^1 f(t,x) \Psi_t(1-t,x) \dx \dt.
 \end{split}
\end{equation}
The integrations by parts performed above are valid because of the null boundary
and initial conditions chosen in systems~\eqref{system.z.rho} 
and~\eqref{system.Psi}. Equation~\eqref{eq.lemma.z.rho} is a direct consequence 
of~\eqref{eq.lemma.z.rho.1} since $\Psi_t = \Phi$.
\end{proof}

Let us come back to the proof of Lemma~\ref{lemma.k}. We apply 
Lemma~\ref{lemma.z.rho} to the state $b$. Thus, from~\eqref{system.b} 
and~\eqref{eq.lemma.z.rho} we deduce that:
\begin{equation} \label{eq.lemma.k.1}
 \begin{split}
   \int_0^1 b(1,x) \rho(x) \dx 
   & = \int_0^1 \int_0^1 \Phi(1-t,x) [-aa_x] (t,x) \dx \dt \\
   & = \frac{1}{2} \int_0^1 \int_0^1 \Phi_x(1-t,x) a^2(t,x) \dx \dt.
 \end{split}
\end{equation}
In order to express our projection directly using $u$, we need to eliminate
$a$ from~\eqref{eq.lemma.k.1}. This can easily be done using an elementary 
solution of the heat system. Therefore, we introduce $G$ the solution to:
\begin{equation} \label{def.G}
 \left\{ 
 \begin{aligned}
    G_t - \varepsilon G_{xx} & = 0 && \text{in } (0,1) \times (0,1), \\
    G(t,0) &= 0 && \text{in } (0,1), \\
    G(t,1) &= 0 && \text{in } (0,1), \\
    G(0,x) &= 1 && \text{in } (0,1).
 \end{aligned}
 \right.
\end{equation}
Using the initial condition $a(t=0,\cdot) \equiv 0$ from system~\eqref{system.a}, 
we can expand $a$ as:
\begin{equation} \label{eq.lemma.k.2}
a(t,x) = \int_0^t G(t-s, x) u(s) \dif s.
\end{equation}
Pluging~\eqref{eq.lemma.k.2} into~\eqref{eq.lemma.k.1} yields:
\begin{equation} \label{eq.lemma.k.3}
 \begin{split}
  \int_0^1 b(1,x) \rho(x) \dif x 
  & = \frac{1}{2} \int_0^1 \int_0^1 \Phi_x(1-t) 
   \left( \int_0^t G(t-s_1) u(s_1) \dif s_1 \right) 
   \left( \int_0^t G(t-s_2) u(s_2) \dif s_2 \right) \dt \\
  & = \frac{1}{2} \int_0^1 \int_0^1 u(s_1) u(s_2) 
    \left( \int_{s_1 \vee s_2}^1 \int_0^1 \Phi_x(1-t) G(t-s_1) G(t-s_2) \dif t \right)
    \ds_1 \ds_2.  
 \end{split}
\end{equation}
Finally, equation~\eqref{eq.lemma.k.3} proves~\eqref{eq.lemma.k} with:
\begin{equation} \label{def.k}
  K^\varepsilon(s_1, s_2) 
  = \frac{1}{2} \int_{s_1 \vee s_2}^1 \int_0^1 
  \Phi_x(1-t,x) G(t-s_1, x) G(t-s_2, x) \dx \dt.
\end{equation}
Thus, we have proved Lemma~\ref{lemma.k} and we have a very precise description 
of the kernel that is involved. This kernel depends on the projection profile 
$\rho(x)$ by means of $\Phi$ defined in~\eqref{def.Phi}. This kernel also 
strongly depends on the viscosity $\varepsilon$ which is involded in the 
computation of both $\Phi$ and of the elementary solution $G$.

Moreover, it is clear that $K$ is a symmetric kernel and since all terms are
bounded thanks to the maximum principle, we know that $K \in L^\infty$. In fact,
$K$ is even smoother as we will see later on.

\subsection{Choice of a profile $\rho$}

As we have seen in the introduction, a natural choice in the low viscosity
setting would be $\rho(x) = x - \frac{1}{2}$.
We think that our proof could be adapted to work with this profile. However,
the computations are tough because it does not satisfy null boundary conditions.
Thus, we are going to make a choice which is more intrinsic to the Burgers 
system.

For any fixed control value $\bar{u} \in \R$, we want to compute the 
associated steady state $(\bar{a}(x), \bar{b}(x))$ of systems~\eqref{system.a}
and~\eqref{system.b}. Thus, we solve the following system:
\begin{equation} \label{system.steady}
 \left\{ 
 \begin{aligned}
   - \varepsilon \bar{a}_{xx} & = \bar{u} && \text{in } (0,1), \\
   - \varepsilon \bar{b}_{xx} & = - \bar{a} \bar{a}_x && \text{in } (0,1),
 \end{aligned}
 \right.
\end{equation}
with boundary conditions $\bar{a}(0) = \bar{a}(1) = \bar{b}(0) = \bar{b}(1) = 0$.
Integrating~\eqref{system.steady} with respect to $x$ yields the following 
family of steady states:
\begin{equation} \label{eq.abar.bbar}
  \bar{a}(x) = \frac{1}{2\varepsilon} x (1-x) \bar{u} 
  \quad \text{and} \quad
  \bar{b}(x) = \frac{1}{8\varepsilon^3} 
  \left(\frac{x^5}{5}-\frac{x^4}{2}+\frac{x^3}{3}-\frac{x}{30}\right) \bar{u}^2.
\end{equation}
Of course, $\bar{b}$ depends quadratically on $\bar{u}$. Thus 
equation~\eqref{eq.abar.bbar} gives the idea of considering:
\begin{equation} \label{def.rho}
  \rho(x) = \frac{x^5}{5}-\frac{x^4}{2}+\frac{x^3}{3}-\frac{x}{30}.
\end{equation}
This choice of $\rho$ may seem strange because is has been obtained using an
infinite viscosity limit. However, since both $\rho$ and $\rho_{xx}$ satisfy
null boundary conditions, the computations of the different kernel residues
turn out to be easier. In the sequel, we assume that $\rho$ is defined
by~\eqref{def.rho}.

\subsection{Rough computation of the asymptotic kernel}

In this paragraph, we apply Lemma~\ref{lemma.k} to compute the kernel associated
to the choice of $\rho$ given in~\eqref{def.rho}. More specifically, we are 
interested in computing a rough approximation of $K^\varepsilon$ when 
$\varepsilon \rightarrow 0$. This approximation will serve as a motivation for 
the following sections where we will need to estimate all the residues that will
be leaving aside for the moment. Since formula~\eqref{def.k} defining 
$K^\varepsilon$ involves both $\Phi$ and $G$, we need to choose approximations
of these quantities as $\varepsilon \rightarrow 0$.
Looking at system~\eqref{def.Phi} defining $\Phi$, we choose to use:
\begin{equation} \label{eq.approx.Phi}
 \Phi_x(t,x) \approx \rho_x(x).
\end{equation}
Moreover, for $G$ defined by~\eqref{def.G}, we will use the approximation 
$G \approx 1$ inside $(0,1)$. Stopping here would not yield anything useful.
Indeed, since $\int_0^1 \rho_x = \rho(1) - \rho(0) = 0$, we would obtain 
$K^\varepsilon = 0$. Hence, we need to choose an approximation of $G$ that is
more accurate near the boundary, eg:
\begin{equation} \label{eq.approx.G}
 G(t,x) \approx \erf \left( \frac{x}{\sqrt{4\varepsilon t}} \right),
\end{equation}
which we will use near $x = 0$. Note that
equation~\eqref{eq.approx.G} corresponds to the solution of a heat equation
on the real line with an initial data equal to $-1$ for $x < 0$ and $+1$ for 
$x>0$. Thus, it satisfies the boundary condition $G(t,0) \equiv 0$ and serves
as a boundary layer correction.
We compute the integrand inside equation~\eqref{def.k}:
\begin{equation} \label{eq.rough.k.1}
  \begin{alignedat}{2}
    A^\varepsilon(t,s_1,s_2)
    & := \frac{1}{2}
      \int_0^1 \Phi_x(1-t,x) G(t-s_1, x) G(t-s_2, x) \dx 
      && \\
    & =
      \frac{1}{2} 
      \int_0^1 \Phi_x(1-t,x) \left( G(t-s_1, x) G(t-s_2, x) - 1 \right) \dx 
      && \text{since $\int \Phi_x = 0$} \\
    & =
      \int_0^\frac{1}{2} \Phi_x(1-t,x) \left( G(t-s_1, x) G(t-s_2, x) - 1\right) \dx 
      && \text{by parity,} \\
    & \approx
      \int_0^\frac{1}{2} \rho_x(x) \left( 
        \erf \left( \frac{x}{\sqrt{4\varepsilon (t-s_1)}} \right)
        \erf \left( \frac{x}{\sqrt{4\varepsilon (t-s_2)}} \right) - 1 \right) \dx 
      && \text{using~\eqref{eq.approx.Phi},~\eqref{eq.approx.G},} \\
    & \approx
      2\sqrt{\varepsilon} \int_0^\frac{1}{4\sqrt{\varepsilon}} 
        \rho_x\left(2 \sqrt{\varepsilon}x\right) \left( 
        \erf \left( \frac{x}{\sqrt{(t-s_1)}} \right)
        \erf \left( \frac{x}{\sqrt{(t-s_2)}} \right) - 1 \right) \dx 
      && \\
    & \sim
      - 2 \sqrt{\varepsilon} \rho_x(0) \int_0^{+\infty} 
        \left( 1 -
        \erf \left( \frac{x}{\sqrt{(t-s_1)}} \right)
        \erf \left( \frac{x}{\sqrt{(t-s_2)}} \right) \right) \dx.
      && 
  \end{alignedat}
\end{equation}
To carry on with the computation, we need the following integral calculus lemma.
\begin{lemma} \label{lemma.erf}
Let $\alpha, \beta > 0$. Then,
\begin{equation} \label{eq.lemma.erf}
\int_0^{+\infty} \left( 1 - \erf (\alpha x) \erf (\beta x) \right) \dif x 
= \frac{1}{\alpha \beta}\sqrt{\frac{\alpha^2 + \beta^2}{\pi}}.
\end{equation}
\end{lemma}

\begin{proof}
We can find an explicit primitive for the integrand. Indeed, for any $X > 0$,
\begin{equation} \label{eq.erf.prim}
 \begin{split}
  \int_0^X \left( 1 - \erf (\alpha x) \erf (\beta x) \right) \dif x 
  = & X \left( 1 - \erf(\alpha X) \erf(\beta X) \right) \\
  & - \frac{\erf(\alpha X) \exp(-\beta^2 X^2)}{\beta\sqrt{\pi}}
    - \frac{\erf(\beta X) \exp(-\alpha^2 X^2)}{\alpha\sqrt{\pi}} \\
  & + \frac{\sqrt{\alpha^2 + \beta^2}}{\alpha\beta\sqrt{\pi}} 
      \erf\left(\sqrt{\alpha^2 + \beta^2} X\right).
 \end{split}
\end{equation}
Equation~\eqref{eq.erf.prim} can be checked by differentiation.
Taking its limit as $X \rightarrow + \infty$ yields~\eqref{eq.lemma.erf}.
\end{proof}

We return to the computation of the asymptotic kernel as 
$\varepsilon \rightarrow 0$. We note that $\rho_x(0) = - \frac{1}{30}$.
Combined with~\eqref{def.k},~\eqref{eq.rough.k.1} and
Lemma~\ref{lemma.erf}, we obtain:
\begin{equation} \label{eq.rough.k.2}
 \begin{split}
  K^\varepsilon(s_1,s_2) 
  & = \int_{s_1 \vee s_2}^1 A^\varepsilon(t, s_1, s_2) \dt \\
  & \approx \frac{\sqrt{\varepsilon}}{15\sqrt{\pi}} 
    \int_{s_1 \vee s_2}^1 \sqrt{ (t - s_1) + (t - s_2) } \dt \\
  & \approx \frac{\sqrt{\varepsilon}}{45\sqrt{\pi}}
    \cdot \left[ \left(2t-s_1-s_2\right)^{\frac{3}{2}} \right]_{s_1 \vee s_2}^1 \\
  & \approx \frac{\sqrt{\varepsilon}}{45\sqrt{\pi}} K^0(s_1,s_2), \\
 \end{split} 
\end{equation}
where we introduce the asymptotic kernel:
\begin{equation} \label{def.k0}
  K^0(s_1, s_2) = \left(2-s_1-s_2\right)^{3/2} - \left|s_1-s_2\right|^{3/2}.
\end{equation}
At this stage, equation~\eqref{eq.rough.k.2} is not rigorous. The meaning of the 
$\approx$ sign has to be made precise. This is the goal of 
Section~\ref{section.keps.residues} where we prove that this asymptotic formula 
does make sense. Indeed, we estimate the kernel residues between $K^\varepsilon$ 
and $\sqrt{\varepsilon} K^0$. They turn out to be both small (with respect 
to~$\varepsilon$) and smooth (with respect to the spaces on which they
define continuous quadratic forms).

\section{Coercivity of the asymptotic kernel}
\label{section.k0.coercive}

In this section, our goal is to prove the coercivity of the kernel $K^0(x,y)$. 
This is a symmetric real-valued kernel defined on $(0,1) \times (0,1)$. Note 
that, since no confusion is possible, we will use $(x,y)$ instead of $(s_1,s_2)$
for the variables of the kernel to lighten notations of this section. We
will prove the following theorem.

\begin{lemma} \label{lemma.coercivity}
 The integral operator associated to $K^0$ is coercive in the space 
 $H^{-5/4}(0,1)$. There exists $\gamma > 0$ such that, for any 
 $f \in L^2(0,1)$, the following inequality holds:
 \begin{equation}
  \int_0^1 \int_0^1 K^0(x,y) f(x) f(y) \dif x \dif y 
  \geq 
  \gamma \left\| F \right\|^2_{H^{-1/4}(0,1)},
 \end{equation} 
 where $F$ is the primitive of $f$ such that $F(0) = 0$.
\end{lemma}

Thanks to the change of variables $(x,y) \mapsto (1-x,1-y)$, the
kernel $K^0$ behaves exactly like:
\begin{equation}
 N(x,y) = \left(x+y\right)^{3/2}  - \left|x-y\right|^{3/2}.
\end{equation}
In this section, we will thus study the properties of $N$ whose expression is
easier to handle.

\subsection{The kernel $N$ is positive definite}

This section uses results and notions from~\cite{MR747302}. 
We will say that a matrix $A$ is \textit{positive 
semidefinite (psd)} when $\langle Ax | x \rangle \geq 0$ for any $x \in \R^m$. 
We will say that $A$ is \textit{positive definite} if the inequality is 
strict for any $x \neq 0$.
We will say that $A$ is \textit{conditionnaly negative semidefinite (cnsd)} when
$\langle Ax | x \rangle \leq 0$ for any $x$ such that $\sum x_i = 0$.
We will use similar definitions for operators.

\begin{lemma} \label{lemma.n.semipositive}
 For any $f \in L^2(0,1)$,
 \begin{equation}
  \label{eq.lemma.n.semipositive}
  \int_0^1 \int_0^1 N(x,y) f(x) f(y) \dif x \dif y 
  \geq 0.
 \end{equation} 
\end{lemma}

\begin{proof}
All necessary arguments can be found in~\cite[Chapter 3]{MR747302}. 
Indeed, the kernel
$-(x+y)^{3/2}$ is \textit{cnsd.} as is proved 
in~\cite[Corollary 2.11]{MR747302}. Moreover, the kernel $|x-y|^{3/2}$ is
also \textit{cnsd.} (see \cite[Remark 1.10]{MR747302} and 
\cite[Corollary 2.10]{MR747302}). Hence, letting:
\begin{equation}
 \psi(x,y) = - \left(x+y\right)^{3/2} + \left|x-y\right|^{3/2}
\end{equation}
defines a \textit{cnsd.} kernel. Thus, since:
\begin{equation}
 N(x,y) = \psi(x,0) + \psi(y, 0) - \psi(x,y) - \psi(0,0),
\end{equation}
this kernel is \textit{psd.} thanks to \cite[Lemma 2.1]{MR747302}.
This proves inequality~\eqref{eq.lemma.n.semipositive}.
\end{proof}
Even though it is true that the kernels involved in the proof of 
Lemma~\ref{lemma.n.semipositive} are striclty negative (or positive), we 
cannot adapt the proof to prove that $N$ is definite. Indeed, Mercer's theorem
(which allows us to take the step from matrices to continuous kernels) doesn't
preserve strict inequalities. Thus, we have to look for another proof.

\subsection{Some insight and facts}

Our main insight is that the kernel $N$ is made up of two parts. The most
singular one should explain its behavior. Indeed, kernels which can be 
expressed as a function $r\left(|x-y|\right)$ have been extensively studied.
For example, \cite{MR0155161} and \cite{MR0150551} prove asymptotic formulas
for the eigenvalues of the $-\left|x-y\right|^{3/2}$ part of our kernel:
\begin{equation} \label{lambda.n}
 \lambda_n \sim \frac{3\sqrt{2}}{4 \pi^2} \left(\frac{1}{n}\right)^{\frac{5}{2}}.
\end{equation}

Moreover, some papers have also studied the eigenvectors of such kernels.
For example, in~\cite{MR0355697}, one can find asymptotic developments for 
eigenvectors of kernels of the form $\left|x-y\right|^{-\alpha}$, 
where $\alpha \in (0,1)$.

\skipline

Combining the insight that the eigenvectors of $N$ should asymptotically behave
like oscillating sinuses and formula~\eqref{lambda.n}, we expect that it should
be possible to prove Lemma~\ref{lemma.coercivity} by means of such an asymptotic
study. However, we have not been able to prove it using this method. Instead, we
give below a proof based on Riesz potentials.

\subsection{Highlighting the singular part of $N$}

The kernel $N(x,y)$ is rather smooth. In order to prove its coercivity, we will
need to isolate it's most singular part. In the following lemma, we use 
integration by parts twice to show that studying the behavior of $N$ is 
equivalent to studying a more singular kernel. By choosing adequatly the 
primitive, we show that we can also cancel boundary terms.

\begin{lemma} \label{lemma.n.ibp}
Let $f \in L^2(0,1)$ and $F$ be the primitive of $f$ such that $F(1) = 0$. Then:
\begin{equation}
 \left( N f , f \right) = \frac{3}{4}
 \izozo \left(\left(x+y\right)^{-\frac{1}{2}} 
    + \left|x-y\right|^{-\frac{1}{2}}\right) F(x) F(y) \dif x \dif y. 
\end{equation}
\end{lemma}

\begin{proof}
Let $f \in L^2(0,1)$ and $F$ be the primitive of $f$ such that $F(1) = 0$. We
start with:
\begin{equation}
 \begin{split}
  - \izozo & \left|x-y\right|^{\frac{3}{2}} f(x) f(y) \dxy \\
  & = - \int_0^1 f(x) \left\{ \int_0^x (x-y)^{\frac{3}{2}} f(y) \dy 
     + \int_x^1 (y-x)^{\frac{3}{2}} f(y) \dy \right\} \dx \\
  & = F(0) \int_0^1 x^{\frac{3}{2}} f(x) \dx 
     + \frac{3}{2} \izozo \left|x-y\right|^{\frac{1}{2}} \mathrm{sg}(y-x) f(x) F(y) \dxy \\
  & = F(0) \int_0^1 x^{\frac{3}{2}} f(x) \dx 
     + \frac{3}{2} \int_0^1 F(y) \left\{ \int_0^y (y-x)^{\frac{1}{2}} f(x) \dx 
       - \int_y^1 (x-y)^{\frac{1}{2}} f(x) \dx \right\} \dy \\
  & = F(0) \int_0^1 \left( x^{\frac{3}{2}} f(x) - \frac{3}{2} x^{\frac{1}{2}} F(x) \right) \dx
     + \frac{3}{4} \izozo \left| x - y \right|^{-\frac{1}{2}} F(x)F(y) \dxy.
 \end{split}
\end{equation}
We continue with the other half of the kernel $N(x,y)$:
\begin{equation}
 \begin{split}
  \izozo & (x+y)^{\frac{3}{2}} f(x) f(y) \dxy \\
  & = - F(0) \int_0^1 x^{\frac{3}{2}} f(x) \dx 
    - \frac{3}{2} \izozo (x+y)^{\frac{1}{2}} f(x) F(y) \dxy \\
  & = F(0) \int_0^1 \left( \frac{3}{2} x^{\frac{1}{2}} F(x) - x^{\frac{3}{2}} f(x) \right) \dx 
    + \frac{3}{4} \izozo (x+y)^{-\frac{1}{2}} F(x) F(y) \dxy.
 \end{split}
\end{equation}
Summing the two previous equalities proves Lemma~\ref{lemma.n.ibp}.
\end{proof}

\subsection{Riesz potential and fractional laplacian}

In this section, we focus on the most singular part of the kernel. 
We recognize a Riesz potential of order $\frac{1}{2}$. Using the fractional
laplacian, we can compute the quantity as a usual norm.

\begin{lemma} \label{lemma.riesz}
There exists $C > 0$ such that, for any $h \in L^2(0,1)$, 
\begin{equation}
\izozo \left|x-y\right|^{-\frac{1}{2}} h(x) h(y) \dxy
\geq C \left\| h \right\|_{H^{-1/4}(0,1)}^2.
\end{equation}
\end{lemma}

\begin{proof}
\begin{equation}
\begin{split}
\izozo \left|x-y\right|^{-\frac{1}{2}} h(x) h(y) \dxy
& = \int_\R \! \int_\R \left|x-y\right|^{-\frac{1}{2}} h(x) h(y) \dxy \\
& = \left( \left(-\Delta\right)^{-1/4} h, h \right) \\
& = \left( \left(-\Delta\right)^{-1/8} h, \left(-\Delta\right)^{-1/8} h \right) \\
& = \left\| \left(-\Delta\right)^{-1/8} h \right\|^2_{L^2} \\
& = \left\| h \right\|^2_{\dot{H}^{-1/4}} \\
& \geq \left\| h \right\|^2_{H^{-1/4}} \\
\end{split}
\end{equation}
More information on such techniques can be found in~\cite{MR0290095} or
posterior works.
\end{proof}

\subsection{Positivity of the smooth part}

To conclude the proof of Lemma~\ref{lemma.coercivity}, we show that the
\emph{smooth} part of our kernel is of positive type. We could also rely
on smoothness arguments to prove that its behavior doesn't modify the 
asymptotic behavior of eigenvectors and eigenvalues of the singular part.

\begin{lemma} \label{lemma.n.smooth.pos}
For any $h \in L^2(0,1)$, 
\begin{equation}
\izozo \left(x+y\right)^{-\frac{1}{2}} h(x) h(y) \dxy \geq 0.
\end{equation}
\end{lemma}

\begin{proof}
We use definitions and theorems found in~\cite[Chapter 3]{MR747302}. Thanks 
to~\cite[result 1.9, page 69]{MR747302}, the kernel given on $(0,1)^2$ 
by $(x,y) \mapsto x+y$ is \emph{conditionnaly negative semidefinite (cnsd)}. 
Hence, using~\cite[corollary 2.10, page 78]{MR747302}, the kernel given by 
$(x,y) \mapsto \sqrt{x+y}$ is also \emph{cnsd}. 
Eventually,~\cite[exercise 2.21, page 80]{MR747302} proves that the kernel
$(x,y) \mapsto 1/\sqrt{x+y}$ is \emph{positive semidefinite}. This means that,
for any $n > 0$ and any $c_1, \ldots c_n \in \R$ and any $x_1, \ldots x_n \in (0,1)$,
\begin{equation}
\sum_{i=1}^n \sum_{j=1}^n \frac{c_i c_j}{\sqrt{x_i + x_j}} \geq 0.
\end{equation}
Using Mercer's theorem (see~\cite{Mercer1909}), we deduce that, for any 
$h \in L^2(0,1)$, 
\begin{equation}
 \izozo \left(x+y\right)^{-\frac{1}{2}} h(x) h(y) \dxy \geq 0.
\end{equation}
\end{proof}

\subsection{Conclusion of the proof}

Now we can prove Lemma~\ref{lemma.coercivity}. Indeed, combining 
Lemmas~\ref{lemma.n.ibp},~\ref{lemma.riesz} and~\ref{lemma.n.smooth.pos} 
proves that there exists $C>0$ 
such that, for any $f \in L^2(0,1)$,
\begin{equation}
 \left( N f, f \right) \geq C \left\| F \right\|_{H^{-1/4}(0,1)}^2,
\end{equation}
where $F$ is the primitive of $f$ such that $F(1) = 0$. Thanks to the change
of variables already mentionned, the same property holds true for $K_0$
with the symmetrical condition $F(0) = 0$.

\section{Exact computation of the kernel and estimation of residues}
\label{section.keps.residues}

In this section, we give a detailed and rigorous expansion of the main kernel 
$K^\varepsilon$. Our goal is to be able to estimate with precision the size and 
the smoothness of all the residues that build up the difference between the 
asymptotic kernel $\sqrt{\varepsilon}K^0$ and the true kernel. As above, we 
write:
\begin{eqnarray} 
 \label{k.a}
 K^\varepsilon(s_1,s_2) & = & 
 \int_{s_1 \vee s_2}^1 A(t,s_1,s_2) \dif t, 
 \quad \textrm{where} \\ 
 \label{def.a}
 A(t,s_1,s_2) & = & 
 \int_0^\frac{1}{2} \Phi_x(1-t,x) G(t-s_1, x) G(t-s_2, x) \dif x.
\end{eqnarray}
In equations~\eqref{k.a} and~\eqref{def.a}, it is implicit that $A$, $\Phi_x$ 
and $G$ depend on $\varepsilon$. Moreover, in equation~\eqref{def.a}, we use 
the fact that $G$ and $\Phi_x$ are even to write the integral over 
$x \in \left(0,\frac{1}{2}\right)$. This breaks the symmetry but will allow us 
to use a one-sided expansion of $G$, thereby focusing on its behavior 
near $x = 0$. 

\subsection{Smoothness of weakly singular integral operators}

We know that the asymptotic kernel $K^0$ is coercive with respect to the $H^{-5/4}$ 
norm of the control $u$. Thus, in order for the full kernel to remain coercive 
for $\varepsilon > 0$, we need to prove that the residues can be bounded with
the same norm. In this paragraph, we give conditions on a kernel residue $L$ that 
are easy to check and imply that:
\begin{equation} \label{est.L}
 \forall u \in L^2(0,1), \quad
 \left| \langle L u , u \rangle \right| 
 \lesssim \left\| U \right\|_{H^{-1/4}(0,1)}^2,
\end{equation}
where $U$ is the primitive of $u$ such that $U(0) = 0$. In the following 
paragraphs, we will check that these conditions are satisfied by our residues.
We start with the following lemma, which allows us to express 
$\langle Lu , u \rangle$ directly as a function of $U$.

\begin{lemma} \label{lemma.kernel.ibp}
Let $\Gamma$ be the triangular domain $\left\{ (x,y) \in (0,1) \times (0,1)
\textrm{, s.t. } x \leq y \right\}$. Let $L \in W^{2,1}(\Gamma)$. We see
$L$ as the restriction to $\Gamma$ of a symmetric kernel on $(0,1)\times(0,1)$
that is smooth on each triangle but not necessarly accross the first diagonal.
Assume that $L(\cdot, 1) \equiv 0$. Let $u \in L^2(0,1)$ and $U$ be the 
primitive of $u$ such that $U(0) = 0$. Then:
\begin{equation} \label{kernel.ibp}
 \int_\Gamma L(x,y) u(x) u(y) \dx\dy 
 = \int_\Gamma \partial_{12} L(x,y) U(x) U(y) \dx \dy + 
 \frac{1}{2} \int_0^1 \left(\partial_1 L - \partial_2 L\right)(x,x) U^2(x) \dx.
\end{equation}
In equation~\eqref{kernel.ibp}, it is worth to be noted that $\partial_1 L$ and 
$\partial_2 L$ are evaluated on the first diagonal and must thus be computed 
using points within $\Gamma$.
\end{lemma}

\begin{proof}
We use integration by parts and the boundary conditions $U(0) = 0$ and 
$L(\cdot, 1) = 0$.
\begin{equation} \label{kernel.ibp.1}
 \begin{split}
  \int_\Gamma L(x,y) u(x) u(y) \dx\dy
  & = \int_0^1 u(x) \int_x^1 L(x,y) u(y) \dy\dx \\
  & = \int_0^1 u(x) \left( \left[ L(x,y) U(y) \right]^1_x - \int_x^1 \partial_2 L(x,y) U(y) \dy \right) \dx \\
  & = - \int_0^1 L(x,x) U(x) u(x) \dx - \int_0^1 U(y) \int_0^y \partial_2 L(x,y) u(x) \dx \\
  & = \int_0^1 \frac{\dif}{\dx} \left\{L(x,x)\right\} \cdot \frac{U^2}{2}(x) \dx \\
  & \quad - \int_0^1 U(y) \left( \left[ U(x) \partial_2 L(x,y) \right]_0^y - 
      \int_0^y \partial_{12} L(x,y) U(x) \dx \right) \dy \\
  & = \int_\Gamma \partial_{12} L(x,y) U(x) U(y) \dx\dy + 
  \frac{1}{2} \int_0^1 \left(\partial_1 L - \partial_2 L\right)(x,x) U^2(x) \dx.
 \end{split}
\end{equation}
Equation chain~\eqref{kernel.ibp.1} concludes the proof of 
equation~\eqref{kernel.ibp}.
\end{proof}

Equation~\eqref{kernel.ibp} includes a boundary term evaluated on the diagonal,
which looks like the $L^2$ norm of $U$. This would forbid us to 
prove any estimate like~\eqref{est.L}. However, all our kernel residues satisfy 
the condition $\partial_1 L - \partial_2 L = 0$ along the diagonal and this
term thus vanishes. Hence, our task is to check that the new kernel 
$\partial_{12}L$ generates a bounded quadratic form on $H^{-1/4}(0,1)$.

\begin{lemma} \label{lemma.wsio}
Let $L$ be a continuous function defined on $\Omega = \left\{ (x,y) \in (0,1) 
\times (0,1) \textrm{, s.t. } x \neq y \right\}$. Assume that there exists 
$\kappa > 0$ and $\frac{1}{2} < \delta \leq 1$, such that, on $\Omega$:
\begin{align}
 \label{est.L.1}
 |L(x,y)| & \leq \kappa |x-y|^{-\frac{1}{2}}, \\
 \label{est.L.2}
 |L(x,y)-L(x',y)| & \leq \kappa |x-x'|^\delta |x-y|^{-\frac{1}{2}-\delta},
 \quad \textrm{for } |x-x'|\leq\frac{1}{2}|x-y|, \\
 \label{est.L.3}
 |L(x,y)-L(x,y')| & \leq \kappa |y-y'|^\delta |x-y|^{-\frac{1}{2}-\delta},
 \quad \textrm{for } |y-y'|\leq\frac{1}{2}|x-y|.
\end{align}
Then $L$ defines a continuous quadractic form on $H^{-1/4}(0,1)$. Moreover, there
exists a constant $C(\delta)$ depending only on $\delta$ (and not on $L$) such
that, for any $U \in L^2(0,1)$:
\begin{equation}
 \label{est.L.final}
 \left| \langle LU, U \rangle \right| \leq C(\delta) \kappa |U|_{H^{-1/4}(0,1)}^2.
\end{equation}
\end{lemma}

This technical lemma is very important for our proof because it gives a 
quantitative estimate, through $\kappa$, of the action of kernels against 
controls. This Lemma can be deduced from the works of Torres~\cite{MR1048075} 
and Youssfi~\cite{MR1397491}. We give a proof skeleton in Appendix~\ref{appendix.wsio}.
The starting point is to prove that a kernel satisfying 
estimates~\eqref{est.L.1},~\eqref{est.L.2} and ~\eqref{est.L.3} defines a 
weakly singular integral operator, which is continuous from $H^{-1/4}$ to
$H^{+1/4}$. Indeed, such kernels are smoother then standard C\'alderon-Zygmund
operators and it is reasonable to expect that they exhibit some smoothing 
properties.

\skipline
We end this section with two useful formulas. Let $a : (0,1)^3 \rightarrow \R$ 
be a function such that  $a(t,s_1,s_2) = a(t,s_2,s_1)$. We consider the kernel
generated by $a$:
\begin{equation} \label{def.k.from.a}
 L(s_1,s_2) = \int_{s_1 \vee s_2}^1 a(t,s_1,s_2) \dt.
\end{equation}
Lemma~\ref{lemma.kernel.ibp} can be applied to such kernels because they satisfy
the condition $L(\cdot, 1) \equiv 0$. We compute:
\begin{align}
 \label{L.a.1}
  \partial_1 L(s,s) - \partial_2 L(s,s) & = a(s,s,s),
  && \textrm{for }s \in (0,1), \\
 \label{L.a.2}
 \partial_{12} L(s_1,s_2) & = - \partial_{s_1} a(s_2, s_1, s_2) 
    + \int_{s_2}^T \partial_{s_1} \partial_{s_2} a(t,s_1,s_2) \dt,
  && \textrm{for }s_1 < s_2.
\end{align}
Formulas~\eqref{L.a.1} and~\eqref{L.a.2} will be used extensively in the 
following sections. Moreover, as soon as $a(s,s,s) \equiv 0$, we see that
the boundary term $\partial_1 L - \partial_2 L$ vanishes.

\subsection{Asymptotic expansion of $K^\varepsilon$}

In this section, we make our rough expansions more precise. Therefore we
decompose $G$ and $\Phi$ using the same first order terms as for the 
heuristic, but this time we introduce and compute the residues.

\subsubsection{Expansion of $G$ as $\varepsilon \rightarrow 0$}

Recall that we only need to approximate $G$ for $x \in (0,1/2)$. Keeping our
approximation introduced in~\eqref{eq.approx.G}, we expand $G$ as:
\begin{equation} \label{eq.G.expansion}
 G(t,x) = \erf\left(\frac{x}{\sqrt{4\varepsilon t}}\right) + H(t, x),
\end{equation}
where $H \in \mathcal{C}^\infty((0,1)\times(0,1/2))$ is the solution to:
\begin{equation} \label{def.H}
 \left\{ 
 \begin{aligned}
    H_t - \varepsilon H_{xx} & = 0
    && \text{in } (0,1) \times (0,1/2), \\
    H(t,0) &= 0 && \text{in } (0,1), \\
    H_x(t,1/2) &= \sigma(\varepsilon t)
    && \text{in } (0,1), \\
    H(0,x) &= 0 && \text{in } (0,1/2),
 \end{aligned}
 \right.
\end{equation}
where the source term $\sigma$ comes from the boundary condition 
$G_x(t,1/2) = 0$ and balances out the trace of the $\erf()$ part:
\begin{equation} \label{def.sigma}
 \sigma(s) 
 = - \left. \frac{\partial}{\partial x} 
  \left[ \erf\left( \frac{x}{\sqrt{4s}} \right) \right]
  \right|_{x=\frac{1}{2}}
 = - \frac{1}{\sqrt{s\pi}} \exp\left(-\frac{1}{16s}\right).
\end{equation}

\begin{lemma} \label{lemma.H}
 Let $0 < \gamma < \frac{1}{16}$. There exists $C(\gamma) > 0$ such that:
 \begin{equation} \label{estimate.lemma.H}
  \left\| H_{t} \right\|_\infty 
  + \left\| H_{t x} \right\|_\infty 
  + \left\| H_{t t} \right\|_\infty 
  + \left\| H_{t t x} \right\|_\infty 
  \leq C(\gamma) e^{-\gamma/\varepsilon}. 
 \end{equation}
\end{lemma}

\begin{proof}
 This lemma is due to the exponentially decaying factor within the source term
 $\sigma$ defined by~\eqref{def.sigma}, which allows as many differentiations
 with respect to $x$ or $t$ as needed to be done. 
 Estimate~\eqref{estimate.lemma.H} could in fact be derived for further 
 derivatives. Let us give a sketch of proof.
 
 First, note that $H^{(3)} := H_{ttt}$ is the solution to a similar 
 system as~\eqref{def.H} with the boundary condition
 $H^{(3)}_x(t, 1/2) = \varepsilon^3 \sigma^{(3)}(\varepsilon t)$. We can 
 convert this boundary condition into a source term by writing
 $H^{(3)}(t, x) = x \varepsilon^3 \sigma^{(3)}(\varepsilon t) + \tilde{H}^{(3)}$,
 where $\tilde{H}^{(3)}$ is now the solution to a heat equation with homogeneous
 mixed boundary conditions and a source term 
 $-x\varepsilon^4 \sigma^{(4)}(\varepsilon t)$.
 Applying the maximum principle yields an estimate of the form 
 $\|\tilde{H}^{(3)}\|_\infty \leq C(\gamma)e^{-\gamma/\varepsilon}$. Since
 $\varepsilon H_{tt xx} = H^{(3)}$, we obtain an $L^\infty$ estimate of 
 the same form for $H_{tt xx}$. By integration with respect to time and
 space, we obtain~\eqref{estimate.lemma.H}.
\end{proof}

\subsubsection{Expansion of $\Phi$ as $\varepsilon \rightarrow 0$}

Guided by our rough computations, we decompose $\Phi \in X_1$, the solution
to~\eqref{def.Phi} as:
\begin{equation} \label{eq.Phi.expansion}
 \Phi(t,x) = \rho(x) + \varepsilon \phi(t,x).
\end{equation}
Thus, we introduce the partial differential equation satisfied by 
$\phi \in X_1$:
\begin{equation} \label{def.phi}
 \left\{ 
 \begin{aligned}
    \phi_t - \varepsilon \phi_{xx} & = \rho_{xx}
    && \text{in } (0,1) \times (0,1), \\
    \phi(t,0) &= 0 && \text{in } (0,1), \\
    \phi(t,1) &= 0 && \text{in } (0,1), \\
    \phi(0,x) &= 0 && \text{in } (0,1).
 \end{aligned}
 \right.
\end{equation}

\begin{lemma} \label{lemma.phi}
The following estimates hold:
\begin{align}
  \label{est.phi.2}
  \left\| \Phi_x \right\|_{\infty}
  & \lesssim 1, \\
  \label{est.phi.3}
  \left\| \phi_x \right\|_{\infty}
  & \lesssim 1, \\
  \label{est.phi.4}
  \left\| \Phi_{t x} \right\|_{\infty} 
  = \left\| \varepsilon \phi_{t x} \right\|_{\infty}
  & \lesssim \varepsilon.
\end{align}
\end{lemma}

\begin{proof}
\textbf{Estimates~\eqref{est.phi.2},~\eqref{est.phi.3} and~\eqref{est.phi.4}}
can be proved using a Fourier series decomposition for heat equations. As an 
example, let us prove~\eqref{est.phi.4}. We introduce the basis 
$e_n(x) = \sqrt{2} \sin(n \pi x)$. Since $\phi_t$ is the solution to a heat
equation with initial data $\rho_{xx} \in H^1_0$, we have:
\begin{equation} \label{eq.lemma.phi.1}
 \phi_t(t, x) 
 = \sum_{n = 1}^{+\infty} e^{-\varepsilon n^2 \pi^2 t}
 \langle \rho_{xx}, e_n \rangle e_n(x).
\end{equation}
Thanks to the choice of $\rho$ in~\eqref{def.rho}, we have 
$\rho_{xx}(0) = \rho_{xx}(1) = 0$. Thus,
\begin{equation} \label{eq.lemma.phi.2}
 \langle \rho_{xx}, e_n \rangle 
 = - \frac{1}{n^2\pi^2} \langle \rho_{xxxx}, e_n \rangle
 = \frac{12 \sqrt{2}}{n^3\pi^3} \left((-1)^n-1\right) 
 = \mathcal{O}\left(\frac{1}{n^3}\right).
\end{equation}
Combining equations~\eqref{eq.lemma.phi.1} and~\eqref{eq.lemma.phi.2} yields:
\begin{equation} \label{eq.lemma.phi.3}
 \left\| \phi_{t x} \right\|_\infty
 \leq \sum_{n=1}^{+\infty} n \pi | \langle \rho_{xx} , e_n \rangle |
 \lesssim \sum_{n=1}^{+\infty} \frac{1}{n^2}.
\end{equation}
Equation~\eqref{eq.lemma.phi.3} concludes the proof of~\eqref{est.phi.4}. A
similar method can be applied to prove~\eqref{est.phi.2} and~\eqref{est.phi.3}.
\end{proof}

\subsubsection{Five stages expansion of the full kernel}

Using expansions~\eqref{eq.G.expansion} and~\eqref{eq.Phi.expansion}, and the 
fact that $\int \Phi_x = 0$, we break down the generator $A(t,s_1,s_2)$ into 6 
smaller kernel generators, $A_1$ through $A_6$, defined by:
\begin{align}
 \label{def.a1}
 A_1(t,s_1,s_2) 
 & = \int_0^{\frac{1}{2}} \rho_x(0) 
   \left( \erf\left(\frac{x}{\sqrt{4\varepsilon(t-s_1)}}\right) 
   \erf\left(\frac{x}{\sqrt{4\varepsilon(t-s_2)}}\right) - 1\right) \dx, \\
 \label{def.a2}
 A_2(t,s_1,s_2) 
 & = \int_0^{\frac{1}{2}} (\rho_x(x) - \rho_x(0))
   \left( \erf\left(\frac{x}{\sqrt{4\varepsilon(t-s_1)}}\right) 
   \erf\left(\frac{x}{\sqrt{4\varepsilon(t-s_2)}}\right) - 1\right) \dx, \\
 \label{def.a3}
 A_3(t,s_1,s_2) 
 & = \int_0^{\frac{1}{2}} \varepsilon \phi_x(1-t,x)
   \left( \erf\left(\frac{x}{\sqrt{4\varepsilon(t-s_1)}}\right) 
   \erf\left(\frac{x}{\sqrt{4\varepsilon(t-s_2)}}\right) - 1\right) \dx, \\
 \label{def.a4}
 A_4(t,s_1,s_2) 
 & = \int_0^{\frac{1}{2}} \Phi_x(1-t,x)
   H(t-s_1, x) 
   \erf\left(\frac{x}{\sqrt{4\varepsilon(t-s_2)}}\right) \dx, \\
 \label{def.a5}
 A_5(t,s_1,s_2) 
 & = \int_0^{\frac{1}{2}} \Phi_x(1-t,x)
   H(t-s_2, x) \cdot 
   \erf\left(\frac{x}{\sqrt{4\varepsilon(t-s_1)}}\right) \dx, \\
 \label{def.a6}
 A_6(t,s_1,s_2) 
 & = \int_0^{\frac{1}{2}} \Phi_x(1-t,x)
   H(t-s_1, x) H(t-s_2, x) \dx.
\end{align}
It can be checked that $A$ defined in~\eqref{def.a} is indeed equal to the sum
of $A_1$ through $A_6$. For each $1 \leq i \leq 6$, we consider the 
associated kernel generated by $A_i$:
\begin{equation} \label{def.k.i}
 K_i(t,s_1,s_2) = \int_{s_1 \vee s_2}^T A_i(t,s_1,s_2) \dt.
\end{equation}
A first remark is that, for each $1 \leq i \leq 6$, $A_i(s,s,s) \equiv 0$ on 
$(0,1)$. Thus, equation~\eqref{L.a.1} tells us that there will be no boundary
term involving $|u|_{H^{-1}}$. 

\subsubsection{Proof methodology}

The six following paragraphs are dedicated to estimates for $K_1$ through $K_6$.
In order to organize the computations that will be carried out for each of these
six kernels, we introduce the notations:
\begin{align}
 \label{def.ti}
 T_i(s_1,s_2) 
 & = \frac{\partial A_i}{\partial s_1} (t, s_1, s_2) \rvert_{t = s_2}, \\
 \label{def.qi}
 Q_i(t,s_1,s_2) 
 & = \frac{\partial^2 A_i}{\partial s_1 \partial s_2} (t, s_1, s_2), \\
 \label{def.ri}
 R_i(s_1,s_2) 
 & = \int_{s_2}^1 Q_i(t,s_1,s_2) \dt. 
\end{align}
Using formula~\eqref{L.a.2}, $\partial_{12}K_i = R_i - T_i$.
Therefore, thanks to Lemma~\ref{lemma.wsio} and Lemma~\ref{lemma.kernel.ibp}, 
we need to prove that each $T_i$ and each $R_i$ satisfies the 
conditions~\eqref{est.L.1},~\eqref{est.L.2} and~\eqref{est.L.3}. For a kernel 
$L$, we will denote $\kappa(L)$ the 
associated constant in Lemma~\ref{lemma.wsio}. In the following paragraphs, 
we investigate the behavior of $\kappa(\partial_{12} K_i)$ with respect to 
$\varepsilon$. We end this paragraph with a useful estimation lemma.
\begin{lemma} \label{lemma.exp.int}
  For any $k > 0$ there exists $c_k > 0$ such that, for any $\lambda > 0$,
  for any $\varepsilon > 0$, 
  \begin{equation}
    \int_0^{+\infty} x^k \exp \left( - \frac{x^2}{4\varepsilon\lambda} \right) \dx
    \leq 
    c_k \left(\varepsilon \lambda\right)^{\frac{k+1}{2}}.
  \end{equation}
\end{lemma}

\begin{proof}
Use a change of variables introducing $\tilde{x} = x / \sqrt{4\varepsilon\lambda}.$
\end{proof}

In the following paragraphs, similarly as we use the $\lesssim$ sign, we will 
use the $\approx$ sign to denote equalities that hold up to a numerical constant
(independent on all variables) of which we will not keep track.

\subsection{Handling the $K_1$ kernel}

The kernel $K_1$ contains the main coercive part of $K^\varepsilon$ discovered
in Section~\ref{section.burgers.k0}. Starting from its definition 
in~\eqref{def.a1}, we decompose it using a scaling on $x$:
\begin{equation} \label{eq.k1.1}
 \begin{split}
  A_1(t,s_1,s_2)
  & = \rho_x(0) \int_0^{\frac{1}{2}} \left(
  \erf\left(\frac{x}{\sqrt{4\varepsilon(t-s_1)}}\right) 
     \erf\left(\frac{x}{\sqrt{4\varepsilon(t-s_2)}}\right) - 1\right) \dx \\
  & = \frac{\sqrt{\varepsilon}}{15} 
  \int_0^{\frac{1}{4\sqrt{\varepsilon}}} 
  \left( 1 - 
  \erf\frac{x}{\sqrt{\alpha}}
  \erf\frac{x}{\sqrt{\beta}}
  \right) \dx \\
  & = \frac{\sqrt{\varepsilon}}{15} \int_0^{+\infty}
  \left(1 - 
  \erf\frac{x}{\sqrt{\alpha}}
  \erf\frac{x}{\sqrt{\beta}}
  \right) \dx 
  - \frac{\sqrt{\varepsilon}}{15} \int_{\frac{1}{4\sqrt{\varepsilon}}}^{+\infty}
  \left(1 - 
  \erf\frac{x}{\sqrt{\alpha}} 
  \erf\frac{x}{\sqrt{\beta}}
  \right) \dx.
 \end{split}
\end{equation}
The first integral gives rise to the main coercive part of the kernel 
and has already been computed exactly in Section~\ref{section.burgers.k0}. The 
second part is a residue and has to be taken care of. Let us name it 
$\tilde{A}_1$:
\begin{equation}
 \tilde{A}_1(t,s_1,s_2) 
 = 
 - \frac{\sqrt{\varepsilon}}{15} 
 \int_{\frac{1}{4\sqrt{\varepsilon}}}^{+\infty} 
 \left( \erf\left(\frac{x}{\sqrt{\alpha}}\right) 
 \erf\left(\frac{x}{\sqrt{\beta}}\right) - 1 \right) \dx.
\end{equation}
Therefore, equation~\eqref{eq.k1.1} yields:
\begin{equation}
 K_1(s_1,s_2) 
 = \frac{\sqrt{\varepsilon}}{45\sqrt{\pi}} K^0(s_1,s_2)
 + \tilde{K}_1(s_1,s_2).
\end{equation}

\begin{lemma}
There exist $c > 0$ and $\gamma > 0$ such that, for any $\varepsilon > 0$,
\begin{equation}
 \kappa(\partial_{12}\tilde{K}_1) 
 \leq c \cdot \exp \left( - \frac{\gamma}{\varepsilon} \right),
\end{equation}
where $\kappa(\partial_{12}\tilde{K}_1)$ is the constant associated to the 
weakly singular integral operator $\tilde{K}_1$ in Lemma~\ref{lemma.wsio}.
\end{lemma}

\begin{proof}
 Recalling notations~\eqref{def.ti},~\eqref{def.qi} and~\eqref{def.ri}, we 
 compute:
\begin{align}
 \label{def.t1}
 \tilde{T}_1(s_1,s_2) 
 = \left( \partial_{s_1} \tilde{A}_1 \right) \rvert_{t = s_2}
 & \approx 
  \varepsilon^{1/2} \Delta^{-3/2} 
  \int_{\frac{1}{4\sqrt{\varepsilon}}}^{+\infty} x
  \exp\left(-\frac{x^2}{\Delta}\right) \dx, \\
 \tilde{Q}_1(t,s_1,s_2) 
 = \partial_{s_1}\partial_{s_2} \tilde{A}_1(t,s_1,s_2)
 & \approx \varepsilon^{1/2} (\alpha\beta)^{-3/2}
  \int_{\frac{1}{4\sqrt{\varepsilon}}}^{+\infty} x^2
  \exp\left(-x^2\left(\frac{1}{\alpha}+\frac{1}{\beta}\right)\right) \dx, \\
 \label{def.r1}
 \tilde{R}_1(s_1,s_2) 
 = \int_{s_2}^1 \tilde{Q}_1(t,s_1,s_2)
 & \approx \varepsilon^{1/2} \int_{s_2}^1 (\alpha\beta)^{-3/2}
  \int_{\frac{1}{4\sqrt{\varepsilon}}}^{+\infty} x^2
  \exp\left(-x^2\left(\frac{1}{\alpha}+\frac{1}{\beta}\right)\right) \dx \dt,
\end{align}
 where we introduce $\Delta = s_2 - s_1$, that will also be used in the sequel.
 We claim that both $\tilde{T}_1$ and $\tilde{R}_1$ are $\mathcal{C}^\infty$ 
 kernels on $(0,1)\times(0,1)$. Moreover, all their derivatives are bounded
 by $e^{-\gamma/\varepsilon}$ for any $\gamma < 1/16$, thanks to the 
 exponential terms in~\eqref{def.t1} and~\eqref{def.r1}. We omit the detailed 
 computations in order to focus on the tougher kernels.
\end{proof}

\subsection{Handling the $K_2$ kernel}

Using the definition of $\rho$ given in~\eqref{def.rho}, we rewrite $A_2$
defined in~\eqref{def.a2} as:
\begin{equation}
 \begin{split}
   A_2(t,s_1,s_2) 
   & = \int_0^{\frac{1}{2}} \left( \rho_x(x) - \rho_x(0) \right)
     \erf \left( \frac{x}{\sqrt{4\varepsilon\alpha}} \right)
     \erf \left( \frac{x}{\sqrt{4\varepsilon\beta}} \right) \dx \\
   & = \int_0^{\frac{1}{2}} x^2 (x-1)^2
     \erf \left( \frac{x}{\sqrt{4\varepsilon\alpha}} \right)
     \erf \left( \frac{x}{\sqrt{4\varepsilon\beta}} \right) \dx.
 \end{split}
\end{equation}
\textbf{First part.} Remembering that $\erf(+\infty)=1$, we consider the first
order derivative:
\begin{equation}
 T_2(s_1,s_2) = \left(\partial_{s_1} A_2\right) \rvert_{t=s_2}
 \approx \varepsilon^{-1/2} \Delta^{-3/2} \int_0^{\frac{1}{2}} x^3(x-1)^2 
 \exp\left( - \frac{x^2}{4\varepsilon\Delta} \right) \dx.
\end{equation}
Using Lemma~\ref{lemma.exp.int} and differentiating gives:
\begin{equation} \label{eq.t2.1}
 \begin{split}
 \left| T_2(s_1,s_2) \right| 
 & \lesssim \varepsilon^{3/2} \Delta^{1/2}, \\ 
 \left| \partial_{s_1} T_2(s_1,s_2) \right| 
 & \lesssim \varepsilon^{3/2} \Delta^{-1/2}, \\
 \left| \partial_{s_2} T_2(s_1,s_2) \right| 
 & \lesssim \varepsilon^{3/2} \Delta^{-1/2}.
 \end{split}
\end{equation}
Estimates~\eqref{eq.t2.1} prove that $\kappa(T_2) \lesssim \varepsilon^{3/2}$.
In fact, $T_2$ is a smoother than the weakly singular integral operators 
studied in Lemma~\ref{lemma.wsio}, since such operators allow degeneracy like
$\Delta^{-1/2}$ along the diagonal. Moreover, we proved that $T_2$ is Lipschitz 
continuous, whereas Lemma~\ref{lemma.wsio} only requires $\mathcal{C}^{p}$ 
with $p > \frac{1}{2}$.

\textbf{Second part.} Now we consider the second order derivative. Let us 
compute:
\begin{equation} \label{def.q12}
 Q_2(t,s_1,s_2) = \partial_{s_1} \partial_{s_2} A_2(t,s_1,s_2) 
 \approx \varepsilon^{-1} \left( \alpha \beta \right)^{-3/2}
 \int_0^{\frac{1}{2}} x^4 (x-1)^2 \exp \left( - \frac{x^2}{4\varepsilon} 
 \left( \frac{1}{\alpha} + \frac{1}{\beta} \right) \right) \dx.
\end{equation}
Thanks to Lemma~\ref{lemma.exp.int}, we estimate the size of $Q_2$:
\begin{equation}
 \left| Q_2(t,s_1,s_2) \right|
 \lesssim
 \varepsilon^{3/2} \left(\alpha\beta\right)^{-3/2}
 \left( \frac{1}{\alpha} + \frac{1}{\beta} \right)^{-5/2}
 = \frac{\varepsilon^{3/2} \alpha\beta}{\left(\alpha+\beta\right)^{5/2}}.
\end{equation}
Writing $\alpha = \Delta + \tau$ and $\beta = \tau$, we can estimate:
\begin{equation} \label{estimate.r12.i}
 \left| R_2(s_1,s_2) \right| 
 = \left|\int_{s_2}^1 Q_2(t,s_1,s_2) \dt\right|
 \lesssim \varepsilon^{3/2} \int_0^1 
   \frac{\tau(\Delta+\tau)}{(\Delta + 2\tau)^{5/2}} \dtau
 \lesssim \varepsilon^{3/2} \Delta^{-1/2}.
\end{equation}
We should now move on to computing $\partial_{s_1}R_2$ and $\partial_{s_2}R_2$, 
to establish the missing estimates on $R_2$. However, the computations 
associated to $R_2$ are very similar to the ones that we carry out for $R_3$. 
Since $R_3$ is a little harder, we skip the proof for $R_2$ and refer the reader 
to the proof of $R_3$, which is fully detailed in the next paragraph. Therefore,
we claim that:
\begin{equation}
 \kappa(\partial_{12}K_2) \lesssim \varepsilon^{3/2}.
\end{equation}

\subsection{Handling the $K_3$ kernel}

In this section, we consider:
\begin{equation}
 \tag{\ref{def.a3}}
 A_3(t,s_1,s_2) = \varepsilon \int_0^{\frac{1}{2}} \phi_x(1-t,x) 
 \left( 
  \erf\left(\frac{x}{\sqrt{4\varepsilon(t-s_1)}}\right)
  \erf\left(\frac{x}{\sqrt{4\varepsilon(t-s_2)}}\right) - 1 \right) \dx.
\end{equation}
\textbf{First part.} Remembering that $\erf(+\infty) = 1$, we consider the 
first order derivative:
\begin{equation}
 T_3(s_1, s_2) := \left( \partial_{s_1} A_3 \right)\rvert_{t = s_2} 
 \approx
 \varepsilon^{1/2} \Delta^{-3/2}
 \int_0^{\frac{1}{2}} \phi_x(1-s_2,x) \cdot
 x \exp \left( - \frac{x^2}{4\varepsilon\Delta} \right) \dx .
\end{equation}
Thanks to Lemma~\ref{lemma.phi} and Lemma~\ref{lemma.exp.int}, we have:
\begin{equation} \label{t13.estimate.i}
 \left| T_3(s_1,s_2) \right|
 \lesssim 
 \varepsilon^{1/2} \Delta^{-3/2} \left\| \phi_{x} \right\|_{\infty} \cdot
 \int_0^{\frac{1}{2}} x \exp \left( - \frac{x^2}{4\varepsilon\Delta} \right) \dx
 \lesssim 
 \varepsilon^{3/2} \Delta^{-1/2}.
\end{equation}
Moreover, 
\begin{equation} \label{t13.estimate.ii}
 \begin{split}
 \left| \partial_{s_1} T_3(s_1, s_2) \right|
 \lesssim & \quad
 \varepsilon^{1/2} \Delta^{-5/2} \left\| \phi_{x} \right\|_{\infty} \cdot
 \int_0^{\frac{1}{2}} x \exp \left( - \frac{x^2}{4\varepsilon\Delta} \right) \dx \\
 & + \varepsilon^{1/2} \Delta^{-3/2} \left\| \phi_{x} \right\|_{\infty} \cdot
 \int_0^{\frac{1}{2}} \frac{x^3}{4\varepsilon\Delta^2} 
 \exp \left( - \frac{x^2}{4\varepsilon\Delta} \right) \dx \\
 \lesssim & \quad \varepsilon^{3/2} \Delta^{-3/2}.
 \end{split}
\end{equation}
and
\begin{equation} \label{t13.estimate.iii}
 \begin{split}
 \left| \partial_{s_2} T_3(s_1, s_2) \right|
 \lesssim 
 & \quad \varepsilon^{1/2} \Delta^{-3/2} \left\| \phi_{xt} \right\|_{\infty} \cdot
 \int_0^{\frac{1}{2}} x \exp \left( - \frac{x^2}{4\varepsilon\Delta} \right) \dx \\
 & + 
 \varepsilon^{1/2} \Delta^{-5/2} \left\| \phi_{x} \right\|_{\infty} \cdot
 \int_0^{\frac{1}{2}} x \exp \left( - \frac{x^2}{4\varepsilon\Delta} \right) \dx \\
 & + \varepsilon^{1/2} \Delta^{-3/2} \left\| \phi_{x} \right\|_{\infty} \cdot
 \int_0^{\frac{1}{2}} \frac{x^3}{4\varepsilon\Delta^2} 
 \exp \left( - \frac{x^2}{4\varepsilon\Delta} \right) \dx \\
 \lesssim & \quad \varepsilon^{3/2} \Delta^{-3/2}.
 \end{split}
\end{equation}
Putting together estimates~\eqref{t13.estimate.i},~\eqref{t13.estimate.ii} 
and~\eqref{t13.estimate.iii} proves that $\kappa(T_3) \lesssim \varepsilon^{3/2}$.

\textbf{Second part}. Let us move on to the second order derivative part.
We compute:
\begin{equation} \label{def.q13}
 Q_3(t,s_1,s_2) 
 = \partial_{s_1}\partial_{s_2} A_3 
 \approx (\alpha\beta)^{-3/2} \int_0^{\frac{1}{2}} x^2 \phi_x(1-t,x) 
 \exp \left( - \frac{x^2}{4\varepsilon} \left(\frac{1}{\alpha}+\frac{1}{\beta}\right)\right)\dx.
\end{equation}
Combining Lemma~\ref{lemma.exp.int} and Lemma~\ref{lemma.phi} yields:
\begin{equation}
 \left| Q_3(t,s_1,s_2) \right| \lesssim
 \frac{\varepsilon^{3/2}}{(\alpha + \beta)^{3/2}}.
\end{equation}
Writing $\alpha = \Delta + \tau$ and $\beta = \tau$, we can estimate:
\begin{equation} \label{estimate.r13.i}
 \left| R_3(s_1,s_2) \right| 
 = \left|\int_{s_2}^1 Q_3(t,s_1,s_2) \dt\right|
 \lesssim \int_0^1 \left(\frac{\varepsilon}{\Delta + 2\tau}\right)^{3/2} \dtau
 \lesssim \varepsilon^{3/2} \Delta^{-1/2}.
\end{equation}
Now we will prove similar estimates for the first order derivatives of 
$R_3$. Differentiating equation~\eqref{def.q13} with respect to $s_1$ (or 
similarly $\alpha$) yields:
\begin{equation}
 \begin{split}
  \partial_{s_1} Q_3(t, s_1, s_2) 
  \approx & - \frac{3}{2} \alpha^{-5/2} \beta^{-3/2} 
  \int_0^{\frac{1}{2}} x^2 \phi_x(1-t,x) \exp \left( - \frac{x^2}{4\varepsilon} 
  \left(\frac{1}{\alpha}+\frac{1}{\beta}\right)\right)\dx \\
  & + \left(\alpha\beta\right)^{-3/2} \frac{1}{\alpha^2} 
  \int_0^{\frac{1}{2}} \frac{x^4}{4\varepsilon} \phi_x(1-t,x) 
  \exp \left( - \frac{x^2}{4\varepsilon} 
  \left(\frac{1}{\alpha}+\frac{1}{\beta}\right)\right)\dx.    
 \end{split}
\end{equation}
Combining Lemma~\ref{lemma.exp.int} and Lemma~\ref{lemma.phi} gives:
\begin{equation}
 \left| \partial_{s_1} Q_3(t, s_1, s_2) \right|
 \lesssim
 \alpha^{-5/2} \beta^{-3/2} \frac{\varepsilon^{3/2}}{
 \left(\frac{1}{\alpha}+\frac{1}{\beta}\right)^{3/2}}
 + \alpha^{-7/2} \beta^{-3/2} \frac{\varepsilon^{3/2}}{
 \left(\frac{1}{\alpha}+\frac{1}{\beta}\right)^{5/2}} \\
 \lesssim
 \varepsilon^{3/2} \alpha^{-5/2}.
\end{equation}
Integration with respect to $t$ yields an estimate of $\partial_{s_1} R_3$:
\begin{equation}
 \left| \partial_{s_1} R_3(s_1,s_2)  \right|
 \lesssim
 \int_{s_2}^1 \left| \partial_{s_1} Q_3(t, s_1, s_2) \right| \dt
 \lesssim
 \varepsilon^{3/2} 
 \int_{s_2}^1 \frac{\dt}{\alpha^{5/2}}
 \lesssim 
 \varepsilon^{3/2} \Delta^{-3/2}.
\end{equation}
From this, we deduce that:
\begin{equation} \label{estimate.r13.ii}
 \left| R_3(s_1,s_2) - R_3(\tilde{s_1},s_2) \right|
 \lesssim
 \varepsilon^{3/2} \Delta^{-3/2} \left| s_1 - \tilde{s_1} \right|.
\end{equation}
Eventually, we finish with the smoothness of $R_3$ with respect to $s_2$.
We compute the difference for $s_1 < s_2 < \tilde{s_2}$ with 
$\tilde{s_2} - s_2 \leq \frac{1}{2} \left(s_2 - s_1\right)$:
\begin{equation}
 \begin{split}
 \left| R_3(s_1,s_2) - R_3(s_1,\tilde{s_2}) \right|
 & = \left| \int_{s_2}^1 Q_3(t,s_1,s_2) \dt - 
            \int_{\tilde{s_2}}^1 Q_3(t,s_1,\tilde{s_2}) \dt \right| \\
 & = \left| \int_{s_2}^{\tilde{s_2}} Q_3(t,s_1,s_2) \dt - 
            \int_{\tilde{s_2}}^1 \left( Q_3(t,s_1,\tilde{s_2}) 
                                      - Q_3(t,s_1,s_2) \right) \dt \right| \\
 & \leq \int_{s_2}^{\tilde{s_2}} \frac{\varepsilon^{3/2}}{\Delta^{3/2}} \dt
  + \left| \int_{\tilde{s_2}}^1 \int_{s_2}^{\tilde{s_2}} 
       \partial_{s_2}Q_3(t,s_1,s) \ds \dt \right| \\
 & \leq \frac{\varepsilon^{3/2}}{\Delta^{3/2}} \left|s_2-\tilde{s_2}\right|
  + \int_{s_2}^{\tilde{s_2}} \int_{\tilde{s_2}}^1 
      \left| \partial_{s_2}Q_3(t,s_1,s) \right| \dt \ds.
 \end{split}
\end{equation}
The first term is already in the correct form. We need to work on the second 
term. Proceeding as above, differentiating equation~\eqref{def.q13} with 
respect to $s_2$ (or similarly $\beta$), then combining 
Lemma~\ref{lemma.exp.int} and Lemma~\ref{lemma.phi} gives:
\begin{equation}
 \left| \partial_{s_2}Q_3(t,s_1,s) \right| 
 \lesssim
 \varepsilon^{3/2} \frac{1}{t-s} \frac{1}{\left(t-s+t-s_1\right)^{3/2}}.
\end{equation}
We compute:
\begin{equation} \label{estimate.r13.iii}
 \begin{split}
  \int_{s_2}^{\tilde{s_2}} \int_{\tilde{s_2}}^1 
    \left| \partial_{s_2}Q_3(t,s_1,s) \right| \dt \ds
  & \leq 
  \varepsilon^{3/2} \int_{s_2}^{\tilde{s_2}} \int_{\tilde{s_2}}^1 
    \frac{1}{t-s}\frac{1}{(t-s_1)^{3/2}} \dt \ds \\
  & \leq 
  \varepsilon^{3/2} \Delta^{-3/2} \int_{s_2}^{\tilde{s_2}} \int_{\tilde{s_2}}^1 
    \frac{\dt}{t-s} \ds \\
  & \leq 
  \varepsilon^{3/2} \Delta^{-3/2} \int_{s_2}^{\tilde{s_2}} 
  \left| \ln\left(\tilde{s_2}-s\right) \right| \ds \\
  & \leq 
  \varepsilon^{3/2} \Delta^{-3/2} \left| s_2 - \tilde{s_2} \right|
  \left( 1 + \ln \left| s_2 - \tilde{s_2} \right| \right).
 \end{split}
\end{equation}
This last estimate does not give Lipschitz smoothness, but it does provide
$\mathcal{C}^p$ smoothness for any $p < 1$, which is enough. Together,
estimates~\eqref{estimate.r13.i},~\eqref{estimate.r13.ii} 
and ~\eqref{estimate.r13.iii} prove that
$\kappa(R_3) \lesssim \varepsilon^{3/2}$. 

\subsection{Handling the $K_4$ kernel}

In this section, we consider:
\begin{equation} \tag{\ref{def.a4}}
 A_4(t,s_1,s_2) = \int_0^{\frac{1}{2}} \Phi_x(1-t,x) H(t-s_1,x) 
 \erf\left(\frac{x}{\sqrt{4\varepsilon(t-s_2)}}\right) \dx.
\end{equation}
\textbf{First part.} We consider the first order derivative:
\begin{equation}
 \begin{split}
 T_4(s_1,s_2) 
 & = \left(\partial_{s_1} A_4\right)\rvert_{t=s_2} \\
 & = \int_0^{\frac{1}{2}} \Phi_x(1-s_2,x) H_{t}(s_2-s_1,x) \dx,
 \end{split}
\end{equation}
where we used the fact that $\erf(+\infty) = 1$. The following estimates
are straight forward:
\begin{align}
 \left| T_2(s_1,s_2) \right| 
 \leq & \linf{\Phi_x} \linf{H_{t}}, \\
 \left| T_2(s_1,s_2) - T_2(\tilde{s}_1,s_2) \right| 
 \leq & \left| s_1 - \tilde{s}_1 \right| \cdot \linf{\Phi_x} \linf{H_{tt}}, \\
 \left| T_2(s_1,s_2) - T_2(s_1,\tilde{s}_2)\right| 
 \leq &\left| s_2 - \tilde{s}_2 \right| \cdot \linf{\Phi_x} \linf{H_{tt}} \\
 & + \left| s_2 - \tilde{s}_2 \right| \cdot \linf{\Phi_{t x}} \linf{H_{t}}.
\end{align}
\textbf{Second part.} We move on to the second order derivative part. We compute:
\begin{equation} \label{def.q2}
 Q_4(t,s_1,s_2) = \partial_{s_1}\partial_{s_2} A_4(t,s_1,s_2)
 \approx 
  \varepsilon^{-1/2} \beta^{-3/2} \int_0^{\frac{1}{2}} 
  x \Phi_x(1-t,x) H_t(\alpha,x) 
  \exp \left( - \frac{x^2}{4\varepsilon\beta} \right) \dx.
\end{equation}
Since $H_t(t,0)\equiv 0$, 
$|H_t(t,x)| \leq x \left\|H_{t x}\right\|_\infty$.
Using Lemma~\ref{lemma.exp.int}, we obtain:
\begin{equation}
 \begin{split}
  \left| Q_4(t,s_1,s_2) \right| 
  & \lesssim \varepsilon^{-1/2} \beta^{-3/2} \left\|H_{t x}\right\|_\infty
  \left\|\Phi_x\right\|_\infty \int_0^{\frac{1}{2}} 
  x^2 \exp \left( - \frac{x^2}{4\varepsilon\beta} \right) \dx \\
  & \lesssim \varepsilon \left\|H_{t x}\right\|_\infty
  \left\|\Phi_x\right\|_\infty.
 \end{split}
\end{equation}
By integration over $t\in(s_2,1)$, we obtain:
\begin{equation}
  \left| R_4(s_1,s_2) \right| 
  \lesssim \varepsilon \left\|H_{t x}\right\|_\infty
  \left\|\Phi_x\right\|_\infty.
\end{equation}
Now we establish the smoothness of $Q_4$ with respect to $s_1$. Differentiating
equation~\eqref{def.q2} with respect to $s_1$ (or $\alpha$), and applying the
same techniques yields the estimate:
\begin{equation}
  \left| \partial_{s_1}Q_4(t, s_1,s_2) \right| 
  \lesssim \varepsilon \left\|H_{t t x}\right\|_\infty
  \left\|\Phi_x\right\|_\infty.
\end{equation}
This proves that:
\begin{equation}
  \left| R_4(s_1,s_2) - R_4(\tilde{s_1},s_2) \right| 
  \lesssim \varepsilon \left\|H_{t t x}\right\|_\infty
  \left\|\Phi_x\right\|_\infty \cdot | s_1 - \tilde{s_1} |.
\end{equation}
Finally, we consider the smoothness of $Q_4$ with respect to $s_2$. We know that:
\begin{equation}
 \left| R_4(s_1,s_2) - R_4(s_1,\tilde{s_2}) \right|
 \leq 
  \int_{s_2}^{\tilde{s_2}} \left| Q_4(t,s_1,s_2) \right| \dt
  + \int_{s_2}^{\tilde{s_2}} \int_{\tilde{s_2}}^1 
      \left| \partial_{s_2}Q_4(t,s_1,s) \right| \dt \ds.
\end{equation}
This first part obviously gives rise to a Lipschitz estimate. As for the second
part, we compute $\partial_{s_2}Q_4$ by differentiating~\eqref{def.q2} with 
respect to $\beta$. We obtain
\begin{equation}
 \begin{split}
  \partial_{s_2}Q_4(t,s_1,s)(t,s_1,s) \approx
  & - \frac{3}{2} \varepsilon^{-1/2} \beta^{-5/2} \int_0^{\frac{1}{2}} 
  x \Phi_x(t,x) H_t(\alpha,x) 
  \exp \left( - \frac{x^2}{4\varepsilon\beta} \right) \dx \\
  & + \varepsilon^{-1/2} \beta^{-3/2} \frac{1}{4\varepsilon\beta^2} \int_0^{\frac{1}{2}} 
  x^3 \Phi_x(t,x) H_t(\alpha,x) 
  \exp \left( - \frac{x^2}{4\varepsilon\beta} \right) \dx.
 \end{split}
\end{equation}
Similar estimates yield:
\begin{equation}
  \left| \partial_{s_2}Q_4(t,s_1,s) \right| 
  \lesssim \varepsilon \left\|H_{t x}\right\|_\infty
  \left\|\Phi_x\right\|_\infty \cdot \frac{1}{t-s}.
\end{equation}
Therefore:
\begin{equation}
 \begin{split}
  \int_{s_2}^{\tilde{s_2}} \int_{\tilde{s_2}}^1  
    \left| \partial_{s_2}Q_4(t,s_1,s) \right| \dt \ds
  & \lesssim 
  \varepsilon \left\|H_{t x}\right\|_\infty \left\|\Phi_x\right\|_\infty \cdot
  \int_{s_2}^{\tilde{s_2}} \int_{\tilde{s_2}}^1 \frac{\dt \ds}{t-s} \\
  & \lesssim 
  \varepsilon \left\|H_{t x}\right\|_\infty \left\|\Phi_x\right\|_\infty \cdot
  \int_{s_2}^{\tilde{s_2}} \left| \ln (\tilde{s_2} - s) \right| \ds \\
  & \lesssim 
  \varepsilon \left\|H_{t x}\right\|_\infty \left\|\Phi_x\right\|_\infty \cdot
  | \tilde{s_2} - s_2 | \left( 1 + \ln  | \tilde{s_2} - s_2 | \right).
 \end{split}
\end{equation}
Therefore, for any fixed $p < 1$, we have:
\begin{equation}
 \left| R_4(s_1,s_2) - R_4(s_1,\tilde{s_2}) \right|
 \lesssim
    \varepsilon \left\|H_{t x}\right\|_\infty \left\|\Phi_x\right\|_\infty \cdot
    \left| \tilde{s_2} - s_2 \right|^p.
\end{equation}
Thanks to Lemma~\ref{lemma.H} and Lemma~\ref{lemma.phi}, this proves that, for 
any $\gamma < \frac{1}{16}$,
\begin{equation}
 \kappa(\partial_{12}K_4) \lesssim \exp\left(-\frac{\gamma}{\varepsilon}\right).
\end{equation}

\subsection{Handling the $K_5$ kernel}

Recall that $A_5$ was defined by:
\begin{equation}
 A_5(t,s_1,s_2) = \int_0^{\frac{1}{2}} \Phi_x(1-t,x) H(t-s_2,x) 
 \erf\left(\frac{x}{\sqrt{4\varepsilon(t-s_1)}}\right) \dx.
 \tag{\ref{def.a5}}
\end{equation}
\textbf{First part.} The first order derivative $T_5$ is null. Indeed,
\begin{equation}
 \begin{split}
 T_5(s_1,s_2) 
 & = \left(\partial_{s_1} A_5\right)\rvert_{t=s_2} \\
 & = \frac{1}{2\sqrt{\pi\varepsilon}} \int_0^{\frac{1}{2}} 
 \Phi_x(1-s_2,x) H(0,x) \cdot \frac{x}{(s_2-s_1)^{\frac{3}{2}}} 
 \exp\left(-\frac{x^2}{4\varepsilon(s_2-s_1)}\right) \dx
 = 0.
 \end{split}
\end{equation}
\textbf{Second part.} We consider the second order derivative:
\begin{equation} \label{def.q5}
 Q_5(t,s_1,s_2) 
 = \partial_{s_2} \partial_{s_1} A_5(t,s_1,s_2) 
 \approx 
 \varepsilon^{-1/2} \alpha^{-3/2}
 \int_0^{\frac{1}{2}} 
   x \Phi_x(t,x) H_{t}(\beta,x) \exp\left(-\frac{x^2}{4\varepsilon\alpha}\right) \dx.
\end{equation}
Since $H_t(t, 0) \equiv 0$, 
$|H_t(t,x)| \leq x \left\|H_{t x}\right\|_\infty$.
Using Lemma~\ref{lemma.exp.int}, we obtain:
\begin{equation}
 \begin{split}
  \left| Q_5(t,s_1,s_2) \right| 
  & \lesssim \varepsilon^{-1/2} \alpha^{-3/2} \left\|H_{t x}\right\|_\infty
  \left\|\Phi_x\right\|_\infty \int_0^{\frac{1}{2}} 
  x^2 \exp \left( - \frac{x^2}{4\varepsilon\alpha} \right) \dx \\
  & \lesssim \varepsilon \left\|H_{t x}\right\|_\infty
  \left\|\Phi_x\right\|_\infty.
 \end{split}
\end{equation}
By integration over $t\in(s_2,1)$, we obtain:
\begin{equation}
  \left| R_5(s_1,s_2) \right| 
  \lesssim \varepsilon \left\|H_{t x}\right\|_\infty
  \left\|\Phi_x\right\|_\infty.
\end{equation}
Differentiating~\eqref{def.q5} with respect to $\alpha$ and proceeding likewise
yields:
\begin{equation}
 \left| \partial_{s_1} Q_5(t,s_1,s_2) \right|
 \lesssim
 \varepsilon \left\|H_{t x}\right\|_\infty
  \left\|\Phi_x\right\|_\infty \cdot \frac{1}\alpha.
\end{equation}
Thus, 
\begin{equation}
 \left| R_5(s_1,s_2) - R_5(\tilde{s_1},s_2) \right|
 \lesssim
 \varepsilon \left\|H_{t x}\right\|_\infty
  \left\|\Phi_x\right\|_\infty \cdot \Delta^{-1} \left| \tilde{s_1} - s_1 \right|.
\end{equation}
Differentiation with respect to $\beta$ is even easier and gives:
\begin{equation}
 \left| \partial_{s_2} Q_5(t,s_1,s_2) \right|
 \lesssim
 \varepsilon \left\|H_{t t x}\right\|_\infty
  \left\|\Phi_x\right\|_\infty,
\end{equation}
from which we easily conclude that $R_5$ is Lipschitz with respect to $s_2$.

Thanks to Lemma~\ref{lemma.H} and Lemma~\ref{lemma.phi}, this proves that, for 
any $\gamma < \frac{1}{16}$,
\begin{equation}
 \kappa(\partial_{12}K_5) \lesssim \exp\left(-\frac{\gamma}{\varepsilon}\right).
\end{equation}

\subsection{Handling the $K_6$ kernel}

Recall that $A_6$ was defined by:
\begin{equation}
 A_6(t,s_1,s_2) 
 = \int_0^{\frac{1}{2}} \Phi_x(1-t,x) H(t-s_1, x) H(t-s_2, x) \dx.
 \tag{\ref{def.a6}}
\end{equation}
\textbf{First part.} The first order derivative $T_6$ is null. Indeed:
\begin{equation}
 T_6(s_1,s_2) = \left( \partial_{s_1} A_6\right) \rvert_{t=s_2} = 
 \int_0^{\frac{1}{2}} \Phi_x(0,x) H_t(s_2-s_1,x) H(0,x) \dx = 0.
\end{equation}
\textbf{Second part.} We consider the second order derivative:
\begin{equation}
 Q_6(t,s_1,s_2) = \partial_{s_2} \partial_{s_1} A_6(t,s_1,s_2) = 
 \int_0^{\frac{1}{2}} \Phi_x(1-s_2,x) H_t(t-s_1,x) H_t(t-s_2,x) \dx.
\end{equation}
For any $t \in (0,1)$, we estimate:
\begin{equation}
 \begin{split}
 \left| Q_6(t,s_1,s_2) \right| 
 \leq & \linf{\Phi_x}\linf{H_{t}}^2, \\
 \left| Q_6(t,s_1,s_2) - Q_6(t,\tilde{s}_1,s_2) \right| 
 \leq & \left| s_1 - \tilde{s}_1 \right| \cdot \linf{\Phi_x} \linf{H_{tt}} \linf{H_{t}}, \\
 \left| Q_6(t,s_1,s_2) - Q_6(t,s_1,\tilde{s}_2)\right| 
 \leq & \left| s_2 - \tilde{s}_2 \right| \cdot \linf{\Phi_x} \linf{H_{t}} \linf{H_{tt}} \\
 & + \left| s_2 - \tilde{s}_2 \right| \cdot \linf{\Phi_{t x}} \linf{H_{t}}^2.
 \end{split}
\end{equation}
Hence, we can extend these estimates to:
\begin{equation}
 R_6(s_1,s_2) = \int_{s_2}^1 Q_6(t,s_1,s_2) \dt
\end{equation}
The only non immediate extension is:
\begin{equation}
 \begin{split}
 \left|R_6(s_1,s_2)-R_6(s_1,\tilde{s}_2)\right|
 \leq & \int_{s_2}^1 \left|Q_6(t,s_1,s_2)-Q_6(t,s_1,\tilde{s}_2)\right| \dt
 + \int_{s_2}^{\tilde{s}_2} \left| Q_6(t,s_1,\tilde{s}_2) \right| \dt \\
 \leq & \left| s_2 - \tilde{s}_2 \right| 
 \left( 
 \linf{\Phi_x} \linf{H_{t}} \linf{H_{tt}} \right. \\
 & \left. + \linf{\Phi_{t x}} \linf{H_{t}}^2
 + \linf{\Phi_x} \linf{H_{t}}^2 \right)
 \end{split}
\end{equation}
Thanks to Lemma~\ref{lemma.H} and Lemma~\ref{lemma.phi}, this proves that, for 
any $\gamma < \frac{1}{16}$,
\begin{equation}
 \kappa(\partial_{12}K_6) \lesssim \exp\left(-\frac{\gamma}{\varepsilon}\right).
\end{equation}

\subsection{Conclusion of the expansion of $K^\varepsilon$}

\begin{lemma} \label{lemma.keps.coercive}
 There exists $\varepsilon_1 > 0$ and $k_1 > 0$ such that, 
 for any $0 < \varepsilon \leq \varepsilon_1$ and any $u \in L^2(0,1)$,
 \begin{equation} \label{eq.k.eps.coercive}
  \langle K^\varepsilon u, u \rangle 
  \geq k_1 \sqrt{\varepsilon} |U|_{H^{-1/4}}^2.
 \end{equation}
\end{lemma}

\begin{proof}
 Thanks to the previous paragraphs, we have shown that
 $K^\varepsilon = \frac{\sqrt{\varepsilon}}{45\sqrt{\pi}} K^0 + R$,
 where $R = \tilde{K}_1 + K_2 + K_3 + K_4 + K_5 + K_6$ is such that
 $\kappa(\partial_{12}R) \lesssim \varepsilon^{3/2}$. From Lemma~\ref{lemma.wsio},
 we deduce that there exists $C_0$ such that, for any $u\in L^2(0,1)$,
 $|\langle Ru , u\rangle| \leq C_0 \varepsilon^{3/2} |U|^2_{H^{-1/4}}$.
 Moreover, thanks to Lemma~\ref{lemma.coercivity}, there exists $c_0$
 such that $\langle K^0 u, u \rangle \geq c_0 |U|^2_{H^{-1/4}}$.
 Hence, for any $k_1 < c_0 / (45\sqrt{\pi})$, 
 equation~\eqref{eq.k.eps.coercive} holds for $\varepsilon$ small enough.
\end{proof}

Equation~\eqref{eq.k.eps.coercive} gives a very weak coercivity, both because
the norm involved is a very weak $H^{-5/4}$ norm on the control $u$, and because
the coercivity constant $k_1 \sqrt{\varepsilon}$ decays when 
$\varepsilon \rightarrow 0$. However, this is enough to overcome the remaining
higher order residues, as we prove in the following section.

\section{Back to the full Burgers non-linear system}
\label{section.back}

In the first part of this work, we studied a second order approximation of our
initial Burgers' system. Thanks to the careful study of an integral kernel, we 
proved that the projection $\langle \rho, b \rangle$ of the state is coercive 
with respect to the control $u$, for a given norm. Now, we want to prove that 
the same fact holds true for the full non-linear system, ie. for the projection
$\langle \rho, y \rangle$. In order to do this, we need to provide estimates 
showing that the projections of the higher order terms in the
expansion of the state are smaller than the coercive quantity obtained above. 
Therefore, we need to prove estimates of $a$, $b$ and the higher order residues
involving the weak $|u|_{H^{-5/4}}$ norm.

\subsection{Preliminary estimates on $a$ and $b$}

\subsubsection{Estimating the first order term $a$}

In order to compute $a$ (defined by system~\eqref{system.a}), a natural idea is 
to introduce $U$ the primitive of $u$ such that $U(0) = 0$. Neglecting the 
impact of the boundary Dirichlet conditions gives the approximation 
$a(t,x) \approx U(t)$. To make this exact, we introduce $\tilde{a}$ which is the
solution to:
\begin{equation} \label{system.a.tilde}
 \left\{ 
 \begin{aligned}
    \tilde{a}_t - \varepsilon \tilde{a}_{xx} & = 0 && \text{in } (0,1) \times (0,1), \\
    \tilde{a}(t,0) &= -U(t) && \text{in } (0,1), \\
    \tilde{a}(t,1) &= -U(t) && \text{in } (0,1), \\
    \tilde{a}(0,x) &= 0 && \text{in } (0,1).
 \end{aligned}
 \right.
\end{equation}
Hence, $a(t,x) = U(t) + \tilde{a}(t,x)$, without any approximation. This 
decomposition is useful because we write $a$ as the sum of a term which does 
not depend on $x$ (thus, $a_x = \tilde{a}_x$) and a term whose size is 
controlled by the desired quantity $|U|_{H^{-1/4}}$. Indeed,

\begin{lemma} \label{lemma.estimates.a}
The following estimates hold:
\begin{align}
  \left\| \tilde{a} \right\|_2 
  & \lesssim \left| U \right|_{H^{-1/4}}, \label{est.a.2} \\
  \left\| a \right\|_{\infty} + \left\| \tilde{a} \right\|_{\infty}
  & \lesssim \left| u \right|_2, \label{est.a.1} \\
  \varepsilon \left\| a_x \right\|_{L^2(L^\infty)} 
  & \lesssim |u|_2. \label{est.a.3}
\end{align}
\end{lemma}

\begin{proof}
The first inequality~\eqref{est.a.2} is a direct application of 
estimate~\eqref{estimate.weak.heat} from Lemma~\ref{lemma.heat.weak}.

The second inequality is a consequence of the maximum principle. Indeed, 
thanks to equation~\eqref{system.a.tilde}, $\left\| \tilde{a} \right\|_{\infty}$
is smaller than $|U|_\infty$. Since $a = U + \tilde{a}$, 
$\left\| a \right\|_{\infty}$ is smaller than 2 $|U|_\infty$. 
Estimate~\eqref{est.a.1} follows because $|U|_\infty \leq |u|_2$.

The third inequality stems from Lemma~\ref{lemma.heat.f}. Since $a$ is 
even, $a_x(\cdot,1/2) \equiv 0$. Thus:
\begin{equation}
  \begin{split}
    \left\| a_x \right\|_{L^2(L^\infty)}^2 
    & = \int_0^1 \left( \sup_{x \in (0,1)} \left| a_x(t,x) \right| \right)^2 \dt \\
    & = \int_0^1 \left( \sup_{x \in (0,1)} 
      \left| \int_{\frac{1}{2}}^x a_{xx}(t,x')\dx' \right| \right)^2 \dt \\
    & \leq \int_0^1 \int_0^1 a_{xx}^2(t,x')\dx' \dt.
  \end{split}
\end{equation}
Combined with~\eqref{estimate.lemma.heat.f}, this proves~\eqref{est.a.3}.
\end{proof}

\subsubsection{Estimating the second order term $b$}

\begin{lemma}
The following estimate holds:
\begin{align} 
  \label{est.b.1}
    \varepsilon^{1/2} \left\| b \right\|_{L^\infty(L^2)} 
      + \varepsilon \left\| b_x \right\|_{L^2}  
    & \lesssim
      \left| u \right|_{L^2} \cdot \left| U \right|_{H^{-1/4}}, \\
  \label{est.b.2}
    \varepsilon^{3/2} \left\| b \right\|_\infty 
    & \lesssim \left| u \right|_2^2, \\
  \label{est.b.3}
    \varepsilon^{3/2} \left\| b_x \right\|_{L^2(L^\infty)}
    & \lesssim \left| u \right|_2^2.
\end{align}
\end{lemma}

\begin{proof}
\textbf{For the first inequality},
we want to apply Lemma~\ref{lemma.heat.fx}. Hence, we want to write the source
term in equation~\eqref{system.b} as a spatial derivative. Writing 
$-aa_x = -\partial_x (a^2/2)$ would not lead to the required estimates.
In order for the weak $H^{-1/4}$ norm to appear, we need to introduce $\tilde{a}$.
Indeed, using the decomposition $a(t,x) = U(t) + \tilde{a}(t,x)$, we can write:
\begin{equation}
- a a_x 
= - a\tilde{a}_x 
= - \frac{\dif}{\dx} \left[ a \tilde{a} - \frac{1}{2} \tilde{a}^2 \right].
\end{equation}
The term under the derivative can easily be estimated in $L^2$:
\begin{equation}
  \left\| a\tilde{a} - \frac{1}{2} \tilde{a}^2 \right\|_{L^2}
  \leq \left\| \tilde{a} \right\|_{L^2} \cdot 
  \left( \left\| a \right\|_{\infty} + \left\| \tilde{a} \right\|_{\infty} \right)
  \lesssim
  \left| u \right|_{L^2} \cdot \left| U \right|_{H^{-1/4}},
\end{equation}
where the last inequality comes from Lemma~\ref{lemma.estimates.a}. Thus, we can
apply Lemma~\ref{lemma.heat.fx} to prove~\eqref{est.b.1}.

\textbf{For the second and third inequalities}, thanks to Lemma~\ref{lemma.heat.f}, 
$\left\|a_x\right\|_2 \lesssim \varepsilon^{-1/2} |u|_2$. Moreover, thanks to
Lemma~\ref{lemma.estimates.a}, $\left\| a \right\|_\infty \lesssim |u|_2$. Thus,
$\left\|a a_x\right\|_2 \lesssim  \varepsilon^{-1/2} |u|_2^2$. We can apply
Lemma~\ref{lemma.heat.f} to show that 
$\left\|b\right\|_{X_1} \lesssim \varepsilon^{-3/2} \left| u \right|_2^2$.
Inequality~\eqref{est.b.2} follows from the injection 
$X_1 \hookrightarrow L^\infty$ (see~\eqref{eq.lemma.xt.c0.2} from 
Lemma~\ref{lemma.xt.c0}).

Moreover, since $\int_0^1 b_x(t,x)\dx = b(t,1) - b(t,0) = 0$ for any 
$t \in (0,1)$, the mean value of $b_x(t, \cdot)$ is 0. Thus, 
$|b_x(t, \cdot)|_\infty \leq |b_{xx}(t,\cdot)|_2$. Hence, 
$ \left\| b_x \right\|_{L^2(L^\infty)} \leq \left\| b_{xx} \right\|_2$.
This proves estimate~\eqref{est.b.3}.
\end{proof}

\subsection{Non-linear residue}

Let us expand $y$ as $a + b + r$, where $a$ stands for the first order linear
approximation, $b$ stands for the second quadratic order and $r$ is a (small) 
residue. Therefore, $r$ is the solution to:
\begin{equation} \label{system.r}
 \left\{ 
 \begin{aligned}
    r_t - \varepsilon r_{xx} & = 
      - rr_x - \left[(a+b)r\right]_x - \left[ ab + \frac{1}{2}b^2 \right]_x 
      && \text{in } (0,1) \times (0,1), \\
    r(t,0) &= 0 && \text{in } (0,1), \\
    r(t,1) &= 0 && \text{in } (0,1), \\
    r(0,x) &= 0 && \text{in } (0,1).
 \end{aligned}
 \right.
\end{equation}

\begin{lemma} \label{lemma.r}
System~\eqref{system.r} admits a unique solution $r \in X_1$. Moreover, 
under the assumption:
\begin{equation} \label{hypothesis.u}
 \left| u \right|_2 \leq \varepsilon^{3/2},
\end{equation}
the following estimate holds:
\begin{equation} \label{est.r}
  \ltwo{r} + \ltwo{r_t} \lesssim \varepsilon^{-3/2} \lu^2 \whu
\end{equation}
\end{lemma}

\begin{proof}
The existence of $r \in X_1$ can be deduced directly from the equality 
$r = y - a - b$. To prove the estimate, we will use 
Lemma~\ref{lemma.burgers.forced} with a null initial data, $w = -(a + b)$
and $g = - ab - \frac{1}{2}b^2$. To apply estimate~\eqref{estimate.burgers.forced},
we start by computing the norms of $w$ and $g$ that we need.
\textbf{We start with $w = - (a + b)$.} Combining~\eqref{est.a.1},~\eqref{est.b.2} 
and~\eqref{hypothesis.u} gives:
\begin{equation} \label{est.r.1}
 \linf{w} 
 \leq \linf{a} + \linf{b} 
 \lesssim \lu + \varepsilon^{-3/2} \lu^2
 \lesssim \lu.
\end{equation}
In particular,~\eqref{est.r.1} and~\eqref{hypothesis.u} yield:
\begin{equation} \label{est.r.2}
 \gamma 
 = \frac{1}{\varepsilon} \size{w}{L^2(L^\infty)}^2 
 \leq \frac{1}{\varepsilon} \linf{w}^2 
 \leq \frac{1}{\varepsilon} \lu^2
 \lesssim 1.
\end{equation}
Finally, combining~\eqref{est.a.3} and~\eqref{est.b.3}:
\begin{equation} \label{est.r.3}
 \size{w_x}{L^2(L^\infty)} 
 \leq \size{a_x}{L^2(L^\infty)} + \size{b_x}{L^2(L^\infty)} 
 \lesssim \varepsilon^{-1} \lu + \varepsilon^{-3/2} \lu^2
 \lesssim \varepsilon^{-1} \lu.
\end{equation}
\textbf{We move on to $g = - ab - \frac{1}{2}b^2$.}
Combining~\eqref{est.a.1},~\eqref{est.b.1},~\eqref{est.b.2} 
and~\eqref{hypothesis.u} gives:
\begin{equation} \label{est.r.4}
 \begin{split}
  \ltwo{g}
  & \leq \left( \linf{a} + \linf{b} \right) \ltwo{b} \\
  & \leq \left( \lu + \varepsilon^{-3/2} \lu^2 \right) 
    \varepsilon^{-1/2} \lu \whu \\
  & \leq \varepsilon^{-1/2} \lu^2 \whu.
 \end{split}
\end{equation}
Combining~\eqref{est.a.1},~\eqref{est.b.1},~\eqref{est.b.2} 
and~\eqref{hypothesis.u}, we obtain:
\begin{equation} \label{est.r.5}
 \begin{split}
  \size{g}{L^2(L^\infty)}
  & \leq \left( \linf{a} + \linf{b} \right) \cdot \size{b}{L^2(L^\infty)} \\
  & \leq \left( \linf{a} + \linf{b} \right) \cdot \ltwo{b_x} \\
  & \lesssim \varepsilon^{-1} \lu^2 \whu.
 \end{split}
\end{equation}
Lastly, mixing~\eqref{est.a.1},~\eqref{est.a.3},~\eqref{est.b.1},~\eqref{est.b.2}
and~\eqref{hypothesis.u} gives:
\begin{equation} \label{est.r.6}
 \begin{split}
  \ltwo{g_x}
  & \leq \size{a_x}{L^2(L^\infty)}\size{b}{L^\infty(L^2)} 
    + \linf{a} \ltwo{b_x} + \linf{b} \ltwo{b_x} \\
  & \lesssim \varepsilon^{-3/2} \lu^2\whu + \varepsilon^{-1}\lu^2\whu
    + \varepsilon^{-5/2} \lu^3 \whu \\
  & \lesssim \varepsilon^{-3/2} \lu^2\whu.
 \end{split}
\end{equation}
\textbf{Eventually,} plugging estimates~\eqref{est.r.1}-\eqref{est.r.6} into 
the main estimation~\eqref{estimate.burgers.forced}, yields:
\begin{equation} \label{est.r.7}
 \ltwo{r_t} \lesssim \varepsilon^{-3/2} \lu^2 \whu.
\end{equation}
From~\eqref{est.r.7} and the initial condition $r(0,\cdot) = 0$, we 
conclude~\eqref{est.r}.
\end{proof}

\begin{lemma} \label{lemma.r.rho}
Under the assumption~\eqref{hypothesis.u}, we have:
\begin{equation} \label{est.r.rho}
  \left| \langle \rho, r(1, \cdot) \rangle \right|
  \lesssim
  \varepsilon^{-3/2} \lu^2 \whu^2.
\end{equation}
\end{lemma}

\begin{proof}
This lemma is not a direct consequence of Lemma~\ref{lemma.r}. Indeed, 
estimate~\eqref{est.r} only involves $|U|_{H^{-1/4}}$ with an exponent
of 1. To obtain estimate~\eqref{est.r.rho}, we need to work a little more.
Using Lemma~\ref{lemma.z.rho} and equation~\eqref{system.r}, we can compute:
\begin{equation} \label{eq.lemma.rrho.1}
  \begin{split}
  \langle \rho, r(1, \cdot) \rangle 
  & = \int_0^1\int_0^1 \Phi_x \left[  ab + \frac{1}{2}b^2 + (a+b)r 
        + \frac{1}{2}r^2 \right]  \\
  & = \int_0^1\int_0^1 \Phi_x(1-t,x) U(t) r(t,x) \dx \dt
    + \int_0^1\int_0^1 \Phi_x \left[ \frac{1}{2}b^2 + (\tilde{a} + b)r 
        + \frac{1}{2}r^2 \right].     
  \end{split}
\end{equation} 
We used the fact that $a = U + \tilde{a}$ and the fact that $\Phi_x a b$
is an odd function, whose space integral is thus null.
The second term is easy to estimate, because we know how to estimate 
$\tilde{a}, b$ and $r$ in $L^2$ using $|U|_{H^{-1/4}}$. Thus, we know it will be 
smaller than $|U|_{H^{-1/4}}^2$. The first term needs more care.
\begin{equation} \label{eq.lemma.rrho.2}
 \int_0^1 U(t) \int_0^1 \Phi_x(1-t,x)r(t,x) \dx \dt = 
 \langle U, v \rangle_{H^{-1}, H^1_0},
\end{equation}
where we introduce $v(t) = \int_0^1 \Phi_x(t,x)r(t,x) \dx$ for $t \in (0,1)$.
Since $\Phi(0, \cdot) \equiv 0$ and $r(0, \cdot) \equiv 0$, $v(0) = v(1) = 0$.
Now we compute its $H^1_0$ norm:
\begin{equation} \label{eq.lemma.rrho.3}
 \begin{split}
  \int_0^1 v_t(t)^2 \dt
  & = \int_0^1 \left( \int_0^1 \Phi_{tx}(1-t,x)r(t,x) 
    + \Phi_x(1-t,x)r_t(t,x) \dx \right)^2 \dt \\
  & \leq 2 \int_0^1 \int_0^1 \Phi_{tx}^2r^2 + \Phi_x^2r_t^2 \\
  & \leq 2 \left( \linf{\Phi_{tx}}^2 \ltwo{r}^2
    + \linf{\Phi_{x}}^2 \ltwo{r_t}^2\right) \\
  & \lesssim \varepsilon^2 \ltwo{r}^2 + \ltwo{r_t}^2 \\
  & \lesssim \ltwo{r_t}^2,
 \end{split}
\end{equation}
where we used estimates~\eqref{est.phi.2} and~\eqref{est.phi.4} to estimate 
$\Phi$. Let us finish the proof.
\begin{equation} \label{eq.rrho}
  \begin{alignedat}{2}
    \left| \langle \rho, r(1, \cdot) \rangle \right|
    & \leq \left| \langle U, v \rangle_{H^{-1}, H^1_0} \right| 
      + \left| \int_0^1 \int_0^1 \Phi_x \left( \frac{1}{2}b^2 
      + (\tilde{a} + b)r + \frac{1}{2}r^2 \right) \right| 
    && \quad \textrm{using~\eqref{eq.lemma.rrho.1} and~\eqref{eq.lemma.rrho.2},} \\
    & \lesssim |U|_{H^{-1}} \ltwo{r_t}
      + \linf{\Phi_x}
      \left( \ltwo{b}^2 + \ltwo{\tilde{a}} \ltwo{r} + \ltwo{r}^2 \right)
    && \quad \textrm{using~\eqref{eq.lemma.rrho.3}}.
  \end{alignedat}
\end{equation}
From~\eqref{est.phi.2}, we know that $\linf{\Phi_x} \lesssim 1$. Moreover,
$|U|_{H^{-1}}\lesssim |U|_{H^{-1/4}}$. Thanks 
to~\eqref{est.a.2},~\eqref{est.b.1},~\eqref{est.r} and~\eqref{hypothesis.u},
we conclude from~\eqref{eq.rrho} that 
$ \left| \langle \rho, r(1, \cdot) \rangle \right|
\lesssim \varepsilon^{-3/2}\lu^2\whu^2$.
This concludes the proof of Lemma~\ref{lemma.r.rho}.
\end{proof}

\subsection{A first drifting result concerning reachability from zero}

The null reachability problem consists in 
computing the set of states that can be reached in time~$T$, starting from 
$y(0,x) \equiv 0$ using a control~$u$. Of course, when dealing with viscous 
equations like~\eqref{system.burgers}, one may only hope to reach sufficiently
smooth states. Theorem~\ref{thm.reachability} tells us that, if the time $T$
is too small, we can never reach a state $y^1(x)$ in time $T$ if 
$\langle \rho, y^1 \rangle < 0$, whatever the control $u$ (and the smoothness
of $y^1$). In some sense, the state \emph{drifts} towards the direction~$+\rho$,
as a result of the action of the control.

\begin{theorem} \label{thm.reachability}
There exist $T_2, k_2 > 0$ such that, for any $0 < T < T_2$
and any $u \in L^2(0,T)$ such that $|u|_{L^2(0,T)}\leq 1$, the solution 
$y \in X_T$ to system~\eqref{system.burgers} starting from the null initial 
condition $y(0,x) \equiv 0$ satisfies:
\begin{equation}
 \langle \rho, y(T,\cdot) \rangle 
 \geq 
 k_2 \left|U\right|_{H^{-1/4}(0,T)}^2,
\end{equation}
where $U$, as above, is the primitive of $u$ such that $U(0) = 0$.
\end{theorem}

\begin{proof}
 We are going to use the scaling argument introduced in 
 paragraph~\ref{subsection.scaling}. Thus, from now on, we reintroduce the 
 tilda signs for functions defined on the scaled time interval $(0,1)$.
 From Lemma~\ref{lemma.keps.coercive}, we know that, for 
 $\varepsilon < \varepsilon_1$, 
 $\langle K^\varepsilon \tilde{u}, \tilde{u} \rangle 
 \geq k_1 \sqrt{\varepsilon} |\tilde{U}|_{H^{-1/4}}^2$.
 From Lemma~\ref{lemma.r.rho}, we know that there exists $c_2$ such that, 
 as soon as $|\tilde{u}|_2 \leq \varepsilon^{3/2}$, 
 $\left|\langle \rho, r(1, \cdot) \rangle\right| 
 \leq c_2 \varepsilon^{3/2} |\tilde{U}|_{H^{-1/4}}^2$.
 Hence, if we consider $\tilde{y}$ the solution to~\eqref{system.y}, write 
 $\tilde{y} = a + b + r$, for any $0 < k_2 < k_1$, there exists 
 $\varepsilon_2 > 0$ such that, for $\varepsilon < \varepsilon_2$, 
 $\langle \rho, \tilde{y}(1,\cdot) \rangle 
 \geq k_2 \sqrt{\varepsilon} |\tilde{U}|_{H^{-1/4}}^2$.
 Recalling that $\tilde{u}(t) = \varepsilon^2 u(\varepsilon t)$ and 
 $\tilde{y}(t,x) = \varepsilon y(\varepsilon t, x)$, we obtain:
 \begin{equation} \label{eq.thm.reach.1}
   \langle \rho, y(\varepsilon,\cdot) \rangle 
   = \left\langle \frac{1}\varepsilon \tilde{y}(1,\cdot), \rho \right\rangle 
   \geq k_2 \varepsilon^{-1/2} |\tilde{U}|_{H^{-1/4}(0,1)}^2
   \geq k_2 |U|_{H^{-1/4}(0,\varepsilon)}^2,
 \end{equation}
 under the assumption:
 \begin{equation} \label{eq.thm.reach.2}
   |\tilde{u}|_{L^2(0,1)} \leq \varepsilon^{3/2}
   \quad \Leftrightarrow \quad
   \left|u\right|_{L^2(0,\varepsilon)} \leq 1.
 \end{equation}
 Theorem~\ref{thm.reachability} follows from~\eqref{eq.thm.reach.1}
 and~\eqref{eq.thm.reach.2} with $T_2 = \varepsilon_2$. 
 Equation~\eqref{eq.thm.reach.2} is obtained via a direct change of variable.
 To establish~\eqref{eq.thm.reach.1}, one can compute the weak $H^{-1/4}$
 norms using Fourier transforms.
\end{proof}

\subsection{Persistance of projections in absence of control}

We start by remarking that, when no control is used, the projection of the state
against any fixed profile $\mu \in L^2(0,1)$ remains almost constant 
in small time.

\begin{lemma} \label{lemma.persistance}
 Let $T > 0$, $\mu \in L^2(0,1)$ and $y^0 \in H^1_0(0,1) \cap H^2(0,1)$. 
 Assume that $|y^0|_{H^2} \leq 1$. Consider $y \in X_T$ the 
 solution to system~\eqref{system.burgers} with initial data $y^0$ and null 
 control ($u = 0$). Then,
 \begin{equation} \label{eq.lemma.persistance}
  \langle \mu, y(T, \cdot) \rangle = 
  \langle \mu, y^0 \rangle + \mathcal{O}\left(T^{1/2} |\mu|_2
  |y^0|_{H^2} \right).
 \end{equation}
\end{lemma}

\begin{proof}
 We decompose $y = y^0 + z$. Hence, $z$ is the solution to:
 \begin{equation} \label{system.lemma.persistance}
 \left\{ 
 \begin{aligned}
    z_t - z_{xx} + z z_x & = (y^0z)_x + y^0_{xx} - y^0 y^0_x
    && \text{in } (0,T) \times (0,1), \\
    z(t,0) &= 0 && \text{in } (0,T), \\
    z(t,1) &= 0 && \text{in } (0,T), \\
    z(0,x) &= 0 && \text{in } (0,1).
 \end{aligned}
 \right.
 \end{equation}
 Thus, we can apply Lemma~\ref{lemma.burgers.forced} with $w(t,x) = y^0(x)$
 and $g(t,x) = y^0_x - \frac{1}{2}(y^0)^2$ 
 to system~\eqref{system.lemma.persistance}. 
 Estimate~\eqref{estimate.burgers.forced} tells us that 
 $\ltwo{z_t} \lesssim |y^0|_{H^2}$. Here, we need the assumption that 
 $|y^0|_{H^2} \leq C$, where $C$ is any fixed constant, in order to avoid
 propagating non-linear estimates (involving exponentials). 
 Since $z(0, x) \equiv 0$, we can write:
 \begin{equation} \label{eq.lemma.pers.1}
  \left| \langle \mu, z(T, \cdot) \rangle \right| 
  = \left| \int_0^T \int_0^1 z_t \mu \right| 
  \leq T^{1/2} \ltwo{z_t} |\mu|_2.
 \end{equation}
 The conclusion~\eqref{eq.lemma.persistance} follows 
 from~\eqref{eq.lemma.pers.1}.
\end{proof}

\subsection{Proof of Theorem~\ref{thm.frederic}}

Let us finish the proof of Theorem~\ref{thm.frederic}. We consider an initial 
data of the form $y^\delta = \delta \rho$, where $\delta > 0$ can be picked as
small as we need and $\rho$ is defined in~\eqref{def.rho}. Please note that 
many other initial data cannot be driven back to zero in short time with small 
controls. However, to prove Theorem~\ref{thm.frederic}, it is sufficient to 
exhibit a single sequence.

For $T > 0$, $u \in L^2(0,T)$ and $\delta > 0$, we consider $y \in X_T$, the 
solution to system~\eqref{system.burgers} with initial data $y^\delta$ and 
control $u$. To isolate the different contributions, we decompose $y$ as
$\bar{y} + y^u + z$, where:
\begin{equation} \label{eq.proof.1}
 \left\{ 
 \begin{aligned}
    \bar{y}_t - \bar{y}_{xx} + \bar{y} \bar{y}_x & = 0
    && \text{in } (0,T) \times (0,1), \\
    \bar{y}(t,0) &= 0 && \text{in } (0,T), \\
    \bar{y}(t,1) &= 0 && \text{in } (0,T), \\
    \bar{y}(0,x) &= y^\delta && \text{in } (0,1),
 \end{aligned}
 \right.
\end{equation}
\begin{equation} \label{eq.proof.2}
 \left\{ 
 \begin{aligned}
    y^u_t - y^u_{xx} + y^u y^u_x & = u(t)
    && \text{in } (0,T) \times (0,1), \\
    y^u(t,0) &= 0 && \text{in } (0,T), \\
    y^u(t,1) &= 0 && \text{in } (0,T), \\
    y^u(0,x) &= 0 && \text{in } (0,1),
 \end{aligned}
 \right.
\end{equation}
\begin{equation} \label{eq.proof.3}
 \left\{ 
 \begin{aligned}
    z_t - z_{xx} + z z_x & = - [(\bar{y}+y^u)z]_x - [\bar{y}y^u]_x
    && \text{in } (0,T) \times (0,1), \\
    z(t,0) &= 0 && \text{in } (0,T), \\
    z(t,1) &= 0 && \text{in } (0,T), \\
    z(0,x) &= 0 && \text{in } (0,1).
 \end{aligned}
 \right.
\end{equation}
Hence, $\bar{y}$ captures the free movement starting from the initial data 
$y^\delta$ while $y^u$ corresponds to the action of the control starting from
a null initial data. Systems~\eqref{eq.proof.1}-\eqref{eq.proof.3} allow us to
decouple these two contributions. The term $z$ is a small residue with 
homogeneous boundary and initial data.

First, let us apply Lemma~\ref{lemma.well-posedness} to 
system~\eqref{eq.proof.1}. Estimates~\eqref{estimate.burgers}
and~\eqref{eq.maximum.burgers} yield:
\begin{equation} \label{eq.proof.4}
 \begin{split}
  \ltwo{\bar{y}_{xx}} + \ltwo{\bar{y}_x} + \ltwo{\bar{y}_t}
  & \lesssim \delta, \\
  \linf{\bar{y}}
  & \leq |y^0|_\infty \lesssim \delta.
 \end{split}
\end{equation}
Similarly, we apply Lemma~\ref{lemma.well-posedness} 
to system~\eqref{eq.proof.2}. If we assume that $\lu \leq 1$ and $T \leq 1$, 
we obtain:
\begin{equation} \label{eq.proof.5}
 \begin{split}
  \ltwo{y^u_{xx}} + \ltwo{y^u_x} + \ltwo{y^u_t}
  & \lesssim \lu, \\
  \linf{y^u}
  & \leq \lu.
 \end{split}
\end{equation}
Next, we look at system~\eqref{eq.proof.3}. We apply 
Lemma~\ref{lemma.burgers.forced} with $w = - (\bar{y} + y^u)$,
$g = - \bar{y} y^u$ and a null initial data. Combining~\eqref{eq.proof.4}
and~\eqref{eq.proof.5} yields the necessary estimates:
\begin{align} 
 \label{eq.proof.6}
 \ltwo{g} + \ltwo{g_x} + \size{g}{L^2(L^\infty)} 
 & \lesssim \delta \lu, \\ 
 \label{eq.proof.7}
 \linf{w} + \size{w}{L^2(L^\infty)} \size{w}{L^2(L^\infty)} 
 & \lesssim \delta + \lu.
\end{align}
Hence,~\eqref{eq.proof.7} yields $\gamma \lesssim 1$. Therefore, 
plugging~\eqref{eq.proof.6} and~\eqref{eq.proof.7} 
into~\eqref{estimate.burgers.forced} gives:
\begin{equation} \label{eq.proof.8}
 \ltwo{z_{xx}} + \ltwo{z_t} \lesssim \delta \lu.
\end{equation}
Once again, we use the initial condition $z(0,\cdot)\equiv 0$ 
and~\eqref{eq.proof.8} to compute:
\begin{equation} \label{eq.proof.9}
 \left| \langle \rho, z(T, \cdot) \rangle \right| 
 = \left| \int_0^T \int_0^1 z_t \rho \right| 
 \lesssim T^{1/2} \delta \lu.
\end{equation}
Now, assuming $T \leq T_2$, we can combine Theorem~\ref{thm.reachability} and 
Lemma~\ref{lemma.persistance} with~\eqref{eq.proof.9} to obtain:
\begin{equation} \label{eq.proof.10}
 \langle y(T, \cdot), \rho \rangle 
 \geq \delta |\rho|^2_2 + k_2 \whu^2 
 + \mathcal{O}\left(T^{1/2}\delta(1+\lu)\right).
\end{equation}
From~\eqref{eq.proof.10}, we deduce that 
$\langle \rho, y(T, \cdot) \rangle > 0$ as soon as $T$ is small enough and under
the assumption $\lu \leq 1$. Thus, we have proved Theorem~\ref{thm.frederic}
with $\eta = 1$.

\section*{Conclusion and perspectives}

We expect that the methodology followed in this paper can be used for a wide
variety of non-linear systems involving a single scalar control. Indeed, when
studying small time local controllability for some formal system 
$\dot{y} = F(y, u(t))$, the first step is always to consider the linearized 
system, $\dot{a} = \partial_y F(0) a + \partial_u F(0) u$. When this system is
controllable, fixed point or inverse mapping theorems often allow us to deduce
that the non-linear system is small time locally controllable. When the 
linearized system is not controllable, we can decompose the state $y$ as 
$a + b$, where the (linear) component $a$ is controllable and the second
component $b$ is indirectly controlled through a quadratic source term involving
$a$ (and/or, sometimes, $u$).

What our proof demonstrates, is that it is possible, even for infinite 
dimensional systems, to express projections of the second order part $b$ as
kernels acting on the control. The careful study of these kernels can then lead
ever to negative results (like it is the case here, because we prove a 
coercivity lemma), or to positive results (if the kernel is found to have both
positive and negative eigenvalues, we can hope to prove that the system can be
driven in the two opposite directions).

It is worth to be noted that the coercivity used in this paper, although it 
involves a weak $H^{-5/4}$ norm of the control $u$, is in fact pretty 
strong. Indeed, it was obtained for any small $u \in L^2$. It would have
been sufficient to prove the coercivity of the kernel $K^\varepsilon$ on the
strict subspace:
\begin{equation}
 \mathcal{V}_\varepsilon = \left\{ 
 u \in L^2(0,1), \quad a(t=1,\cdot) \equiv 0,
 \quad \textrm{where $a$ is the solution to system~\eqref{system.a}}
 \right\}.
\end{equation}
For other systems, it may be easier (or necessary) to restrict the study of the 
integral operator~$K^\varepsilon$ to the subspace~$\mathcal{V}_\varepsilon$ in
order to obtain a conclusion.

\skipline

As a perspective, 
an example of such an open problem is the small time controllability of 
the non-linear Korteweg de Vries equation for critical domains. Indeed, 
in~\cite{MR1440078}, Rosier proved that the KdV equation was small time 
locally controllable for non critical domains using the linearized system.
Then in~\cite{MR2060480}, Coron and Crépeau proved that, for the first critical
length, small time local controllability holds thanks to a third order
expansion. In~\cite{MR2338431} and~\cite{MR2504039}, Cerpa then Cerpa and Crépeau
proved that large time local controllability holds for all critical lengths.
It remains an open question to know whether small time local controllability 
holds for the second critical length. Maybe our method could be adapted to 
this setting or inspire a new proof.

\skipline

The author thanks Sergio Guerrero for having attracted his 
attention on this control problem and his advisor Jean-Michel Coron for his
support and ideas all along the elaboration of this proof.

\appendix

\section{Weakly singular integral operators}
\label{appendix.wsio}

\newcommand{\dsmooth}{\mathcal{D}\left(\R^n\right)}
\newcommand{\dsmoothp}{\mathcal{D}_p\left(\R^n\right)}
\newcommand{\dprime}{\mathcal{D}'\left(\R^n\right)}
\newcommand{\ssmooth}{\mathcal{S}\left(\R^n\right)}
\newcommand{\sprime}{\mathcal{S}'\left(\R^n\right)}
\newcommand{\wsio}{\mathrm{WSIO}(s, \delta)}
\newcommand{\hbesov}[3]{\dot{B}^{#1,#3}_{#2}}

This appendix is devoted to an explanation of Lemma~\ref{lemma.wsio}. Although
a full proof would exceed the scope of this article, we provide here a brief
overview of a general method introduced by Torres in~\cite{MR1048075} to study
the regularization properties of weakly singular integral operators. Our 
presentation is also inspired by a posterior work of Youssfi, who states a very 
closely related lemma in~\cite[Remark 6.a]{MR1397491}.

Let $n \geq 1$. Singular integral operators on $\R^n$ have been extensively 
studied since the seminal works of Calder{\'o}n and Zygmund 
(see~\cite{MR0052553} and~\cite{MR0084633}). These integral operators are
defined by the singularity of their kernel along the diagonal by an estimate
of the form:
\begin{equation} \label{sio.cz}
 \left| K(x,y) \right| \leq C \left|x-y\right|^{-n}.
\end{equation}
In estimate~\eqref{sio.cz}, the exponent $-n$ is critical. Indeed, the margins 
of such kernels are almost in $L^1_\textrm{loc}$. Here, we are interested in a 
class of integral operators for which the singularity along the diagonal is 
weaker. Thus, we expect that they exhibit better smoothing properties. 
Throughout this section, we denote 
$\Omega = \left\{ (x,y) \in \R^n\times\R^n, x \neq y \right\}$.

\begin{definition}[Weakly singular integral operator] \label{def.wsio}
Let $0 < s < 1$ and $0 < \delta \leq 1$. Consider a kernel~$K$, continuous on 
$\Omega$, satisfying:
\begin{gather} 
  \label{wsio.i}
  \left| K(x,y) \right| \leq \kappa \left|x-y\right|^{-n+s}, \\
  \label{wsio.ii}
  \left| K(x',y) - K(x,y) \right| \leq \kappa \left|x'-x\right|^\delta
  \left|x-y\right|^{-n+s-\delta}, 
  \quad \textrm{for } \left|x'-x\right| \leq \frac{1}{2}\left|x-y\right|, \\
  \label{wsio.iii}
  \left| K(x,y') - K(x,y) \right| \leq \kappa \left|y'-y\right|^\delta
  \left|x-y\right|^{-n+s-\delta}, 
  \quad \textrm{for } \left|y'-y\right| \leq \frac{1}{2}\left|x-y\right|.
\end{gather}
We introduce the associated integral operator $T_K$, continuous 
from $\dsmooth$ to $\dprime$, by defining:
\begin{equation} \label{tk.pointwise}
 \forall f \in \dsmooth, \forall x \in \R^n, \enskip 
 T_K(f)(x) = \int K(x,y) f(y) \dy.
\end{equation}
Under these assumptions, we write $T_K \in \wsio$.
\end{definition}

Definition~\ref{def.wsio} can be extended for $s \geq 1$. 
Conditions~\eqref{wsio.i},~\eqref{wsio.ii} and~\eqref{wsio.iii} must then be 
extended to the derivatives $\partial_x^\alpha \partial_y^\beta K$ for 
$\alpha + \beta \leq s$. We restrict ourselves to the simpler setting 
$0 < s < 1$ as it is sufficient for our study. Note that we define the operator
$T_K$ from its kernel $K$ (as this is the case for our applications). Proceeding 
the other way around is possible but would require more care in the sequel 
(namely, the so-called \emph{weak boundedness property} to ensure 
that~\eqref{tk.pointwise} holds;~see~\cite{MR1397491}).

\subsection{Atomic and molecular decompositions for Triebel-Lizorkin spaces}

We recall the definitions of classical functional spaces involved in this 
appendix.
Let $\varphi \in \ssmooth$ be such that $\varphi(\xi) = 0$ for $|\xi|\geq1$ 
and $\varphi(\xi) = 1$ for $|\xi| \leq \frac{1}{2}$. We introduce 
$\psi(\xi) = \varphi(\xi/2) - \varphi(\xi)$. Hence, $\psi \in \ssmooth$
and is supported in the annulus $\{ \frac{1}{2} \leq |\xi| \leq 2 \}$. We will
denote $\dot{\Delta}_j$ and $\dot{S}_j$ the convolution operators with symbols
$\psi(2^{-j}\xi)$ and $\varphi(2^{-j}\xi)$.

\begin{definition}[Homogeneous Besov space] \label{def.besov}
For $\alpha\in\R$, $1 \leq p, q \leq \infty$, the homogeneous Besov space
$\dot{B}^{\alpha,q}_p$ is defined by the finiteness of the norm
(with standard modification for $q = \infty$):
\begin{equation} \label{eq.def.besov}
 \left\| f \right\|_{\dot{B}^{\alpha,q}_p}
  = \left( \sum_{j\in\Z} 2^{\alpha q j} \left\| \dot{\Delta}_j f \right\|_p^q \right)^{1/q}.
\end{equation}
\end{definition}

\begin{definition}[Homogeneous Triebel-Lizorkin space] \label{def.triebel}
For $\alpha\in\R$, $1 \leq p, q < \infty$, the homogeneous Triebel-Lizorkin space
$\dot{F}^{\alpha,q}_p$ is defined by the finiteness of the norm:
\begin{equation} \label{eq.def.triebel}
 \left\| f \right\|_{\dot{F}^{\alpha,q}_p}
  = \left\| \left( 
      \sum_{j\in\Z} 2^{\alpha q j} | \dot{\Delta}_j f |^q
    \right)^{1/q} \right\|_p.
\end{equation}
\end{definition}

Frazier and Jawerth introduced \emph{atoms} and \emph{molecules} both in the
context of Besov spaces (\cite{MR808825}) and Triebel-Lizorkin spaces 
(\cite{MR942271} and~\cite{MR1070037}). They proved that the norms on these 
spaces are then translated into sequential norms on the sequence of coefficients
of the decomposition. A linear operator will be continuous between two 
Triebel-Lizorkin spaces if and only if it maps smooth atoms of the first to 
smooth molecules of the second. The following definitions are borrowed 
from~\cite{MR1048075}. For simplicity, we restrict them to the case 
$1 \leq p, q \leq + \infty$.

\begin{definition}[Smooth atom] \label{def.atom}
Let $\alpha \in \R$ and $Q$ be a dyadic cube in $\R^n$ of side length $\ell_Q$.
A smooth $\alpha$-atom, associated with the cube $Q$ is a function 
$a \in \dsmooth$ satisfying:
\begin{gather}
  \supp (a) \subset 3 Q, \label{atom.i} \\
  \int x^\gamma a(x) \dx = 0, \quad
    \forall |\gamma|\leq \max\{0, [-\alpha]\}, \label{atom.ii} \\
  \left| \partial^\gamma_x a(x) \right| 
    \leq \ell_Q^{-|\gamma|}, \quad
    \forall |\gamma| \leq \max\{0, [\alpha]\} + 1. \label{atom.iii}
\end{gather}
\end{definition}
In condition~\eqref{atom.i}, $3Q$ denotes the cube with same center as $Q$ but 
a tripled side length. It is worth to be noted that multiple normalization 
choices are possible for condition~\eqref{atom.iii}. We choose to only include
the decay corresponding to the smoothness of the atom. This choice only impacts 
the formula to compute the size of a function from its decomposition on atoms.
We have the following representation theorem:

\begin{lemma}[Theorem 5.11, \cite{MR1107300}] \label{lemma.fjw.511}
Let $\alpha \in \R$, $1 \leq p, q < \infty$. Let $f \in \dot{F}^{\alpha,q}_p$.
There exists a sequence of reals $(s_Q)_{Q \in \mathcal{Q}}$ indexed by the
set $\mathcal{Q}$ of dyadic cubes of $\R^n$ and a sequence of atoms 
$(a_Q)_{Q \in \mathcal{Q}}$ such that $f = \sum_{Q} s_Q a_Q$. Moreover, there
exists a constant $C$ independent on $f$ such that:
\begin{equation} \label{eq.fjw.511}
  \left\| \left( 
    \sum_{Q} \ell_Q^{-\alpha q} |s_Q|^q |\chi_Q(x)|^q 
  \right)^{1/q} \right\|_p \leq C \left\| f \right\|_{\dot{F}^{\alpha,q}_p}.
\end{equation}
\end{lemma}

The reciprocal inequality to~\eqref{eq.fjw.511} is true even for a wider class
of functions, the class of molecules.

\begin{definition}[Smooth molecule] \label{def.molecule}
Let $\alpha \in \R$, $M > n$ and $\alpha - [\alpha] < \delta \leq 1$. Let $Q$ 
be a dyadic cube in $\R^n$ of side length $\ell_Q$ and center $x_Q$. A 
$(\delta, M)$ smooth $\alpha$-molecule associated with $Q$ is a function $m$ 
satisfying:
\begin{gather}
  \left| m(x) \right| \leq \left( 1 + \ell_Q^{-1} 
    \left| x - x_Q \right| \right)^{-\max\{M, M-\alpha\}}, \label{molecule.i} \\
  \int x^\gamma m(x) \dx = 0, \quad
    \forall |\gamma| \leq [-\alpha], \label{molecule.ii} \\
  \left| \partial^\gamma_x m(x) \right| \leq
    \ell_Q^{-|\gamma|}
    \left( 1 + \ell_Q^{-1} \left| x - x_Q \right| \right)^{-M}, \quad
    \forall |\gamma| \leq [\alpha], \label{molecule.iii} \\
   \left| \partial^\gamma_x m(x) -  \partial^\gamma_x m(x') \right| \leq
     \ell_Q^{-|\gamma|-\delta} \left| x - x' \right|^\delta
     \sup_{|z| \leq |x-x'|} \left( 1 + 
     \ell_Q^{-1} \left| z - (x - x_Q) \right| \right)^{-M},
     \quad \forall |\gamma| = [\alpha]. \label{molecule.iv}
\end{gather}
\end{definition}
In the definition of a molecule, conditions~\eqref{molecule.iii} 
and~\eqref{molecule.iv} are void by convention if $\alpha < 0$.
When $\alpha \geq 0$, condition~\eqref{molecule.iii} implies~\eqref{molecule.i}.
When $\alpha > 0$, condition~\eqref{molecule.ii} is void. 
We have:

\begin{lemma}[Theorem 5.18, \cite{MR1107300}] \label{lemma.fjw.518}
Let $\alpha \in \R$, $M > n$ and $\alpha - [\alpha] < \delta \leq 1$.
Consider a sequence of reals $(s_Q)_{Q \in \mathcal{Q}}$ indexed by the
set $\mathcal{Q}$ of dyadic cubes of $\R^n$ and a sequence of $(\delta, M)$
smooth $\alpha$-molecules $(m_Q)_{Q \in \mathcal{Q}}$. 
Let $f = \sum_{Q} s_Q m_Q$. There exists a constant $C$ independent on $f$ 
such that:
\begin{equation} \label{eq.fjw.518}
  \left\| f \right\|_{\dot{F}^{\alpha,q}_p}
  \leq C  
  \left\| \left( 
    \sum_{Q} \ell_Q^{-\alpha q} |s_Q|^q |\chi_Q(x)|^q 
  \right)^{1/q} \right\|_p .
\end{equation}
\end{lemma}

\subsection{Circumventing the $T(1) = 0$ condition}

When dealing with singular integral operators, difficulties arise when 
$T(1) \neq 0$. Most regularity results involve some smoothness
condition on $T(1)$ (see, for example the early paper~\cite{MR763911}).
To circumvent this difficulty when handling weakly singular integral operators,
we will write $T_K = \tilde{T}_K + \pi$ where $\tilde{T}_K$ satisfies the
same regularity estimates as $T_K$ but is such that $\tilde{T}_K(1) = 0$
and $\pi$ is defined as a paraproduct, for which we can get direct smoothing
estimates in the appropriate spaces. For two functions $f,g$, we introduce the 
following paraproduct $\pi$, inspired by ideas of J.-M. Bony (see the seminal
work~\cite{MR631751}, the nice introduction to paraproducts~\cite{MR2682821} for
a quick overview or \cite[Section 2.6.1]{MR2768550} for a complete detailed 
presentation):
\begin{equation} \label{def.paraproduct}
 \pi_g(f) = \sum_{j \in \Z} \dot{\Delta}_j(g) \dot{S}_{j-2}(f). 
\end{equation}

\begin{lemma}[Lemma 4, \cite{MR1397491}] \label{lemma.t1}
Let $0 < s < \delta \leq 1$ and $T_K \in \wsio$. Then,
$T_K(1) \in \dot{B}^{s,\infty}_\infty$. Moreover, there exists $C = C(s,\delta)$
such that: $\|T_K(1)\|_{\dot{B}^{s,\infty}_\infty} \leq C \kappa(T_K)$ where 
$\kappa(T_K)$ is the constant associated to $T_K$ in Definition~\ref{def.wsio}.
\end{lemma}

\begin{lemma}[Remark 2, \cite{MR1397491}] \label{lemma.paraprod.continuity}
Let $1 \leq p, q < \infty$, $t < 0$ and $s \in \R$. There exists 
$C = C(p, q, t, s)$ such that, for any $b \in \dot{B}^{s,\infty}_\infty$, 
$\pi_b$ is continuous from $\dot{F}^{t,q}_p$ to $\dot{F}^{t+s,q}_p$ and the
following estimate holds:
\begin{equation} \label{eq.lemma.paraprod.continuity}
 \forall f \in \dot{F}^{t,q}_p,
  \left\| \pi_b(f) \right\|_{\dot{F}^{t+s,q}_p}
  \leq C \left\| b \right\|_{\dot{B}^{s,\infty}_\infty} 
    \left\| f \right\|_{\dot{F}^{t,q}_p}.
\end{equation}
\end{lemma}

\begin{lemma}[Lemma 2, \cite{MR1397491}] \label{lemma.pi.b.wsio}
Let $0 < s < 1$ and $0 < \delta \leq 1$. Take $b \in \dot{B}^{s,\infty}_\infty$.
Then, the operator $\pi_b \in \wsio$. Moreover, there exists a constant $C(s)$ 
independent of $b$ such that, 
$\kappa(\pi_b) \leq C(s) \|b\|_{\dot{B}^{s,\infty}_\infty}$, 
where $\kappa(\pi_b)$ is the constant in Definition~\ref{def.wsio}
associated to the operator $\pi_b$.
\end{lemma}

Combining these lemmas allows us to circumvent the $T(1) = 0$ condition. Indeed:
\begin{lemma} \label{lemma.continuity.tk}
Let $0 < s < \delta \leq 1$ and $1 \leq p, q < \infty$. Let $t \in \R$ be such 
that $-s < t < 0$. There exists a constant $C$ such that, for $T_K \in \wsio$,
$T_K$ is continuous from $\dot{F}^{t,q}_p$ into $\dot{F}^{t+s,q}_p$ and we 
have:
\begin{equation} \label{eq.continuity.tk}
  \forall f \in \dot{F}^{t,q}_p, \quad
  \left\| T_K(f) \right\|_{\dot{F}^{t+s,q}_p} 
  \leq C \kappa(T_K) \left\| f \right\|_{\dot{F}^{t,q}_p}, 
\end{equation}
where $\kappa(T_K)$ is the constant associated to $T_K$ in 
Definition~\ref{def.wsio}.
\end{lemma}

\begin{proof}
Let $T_K \in \wsio$. Thanks to Lemma~\ref{lemma.t1}, 
$T_K(1) \in \dot{B}^{s,\infty}_\infty$ and 
$\|T_K(1)\|_{\dot{B}^{s,\infty}_\infty} \lesssim \kappa(T_K)$.
Thanks to Lemma~\ref{lemma.pi.b.wsio}, $\pi_{T_K(1)} \in \wsio$ and
$\kappa(\pi_{T_K(1)}) \lesssim \kappa(T_K)$. Hence, we can define
$\tilde{T}_K := T_K - \pi_{T_K(1)}$ and $\tilde{T}_K \in \wsio$, with
a constant $\kappa(\tilde{T}_K) \lesssim \kappa(T_K)$. Moreover, since
$\pi_{b}(1) = b$ for any $b$, $\tilde{T}_K(1) = 0$. Thanks to 
Lemma~\ref{lemma.paraprod.continuity}, proving the continuity of 
$\tilde{T}_K$ is sufficient to obtain~\eqref{eq.continuity.tk}.

Let $a_Q$ be a smooth $t$-atom. We consider $m_Q = \tilde{T}_K(a_Q)$. The next 
step is to prove that $m_Q$ is almost a $(\delta, M)$ smooth $(t+s)$-molecule, 
with $M = n + s - \delta > n$. As noted above, since $t+s > 0$, we only need to 
check~\eqref{molecule.iii} and~\eqref{molecule.iv}. Indeed, lengthy computations
and the essential condition $\tilde{T}_K(1) = 0$ provide the existence of a 
constant $D$ independent on the atom $a_Q$ such that:
\begin{gather}
 \left| m_Q(x) \right| \leq D \ell_Q^s
    \left( 1 + \ell_Q^{-1} \left| x - x_Q \right| \right)^{-M}, \label{m.iii} \\
 \left| m_Q(x) -  m_Q(x') \right| \leq D \ell_Q^s
     \ell_Q^{-\delta} \left| x - x' \right|^\delta
     \sup_{|z| \leq |x-x'|} \left( 1 + 
     \ell_Q^{-1} \left| z - (x - x_Q) \right| \right)^{-M}.
      \label{m.iv}
\end{gather}
Hence $\tilde{m}_Q := D^{-1}\ell_Q^{-s} m_Q$ is a molecule.
For examples of proof techniques to prove~\eqref{m.iii} and~\eqref{m.iv}, we
refer the reader to~\cite{MR1048075} and~\cite{MR1397491}. To conclude the 
proof, we use Lemma~\ref{lemma.fjw.511} and~\ref{lemma.fjw.518}. For 
$f \in \dot{F}^{t,q}_p$, we write $f(x) = \sum_Q s_Q a_Q(x)$ and each 
$\tilde{m}_Q = D^{-1}\ell_Q^{-s} T_K(a_Q)$ is a molecule. Thus, thanks to 
Lemma~\ref{lemma.fjw.511} and Lemma~\ref{lemma.fjw.518},
\begin{equation} \label{eq.cont.tk.1}
 \begin{split}
  \left\| T_K(f) \right\|_{\dot{F}^{t+s,q}_p}
  & = \left\| \sum_Q (D \ell_Q^s s_Q) \cdot m_Q(x) \right\|_{\dot{F}^{t+s,q}_p} \\
  & \lesssim \left\| \left( 
    \sum_{Q} \ell_Q^{-(t+s) q} D^q \ell_Q^{sq} |s_Q|^q |\chi_Q(x)|^q 
  \right)^{1/q} \right\|_p \\
  & \lesssim \left\| \left( 
    \sum_{Q} \ell_Q^{-tq} |s_Q|^q |\chi_Q(x)|^q 
  \right)^{1/q} \right\|_p \\
  & \lesssim \left\| f \right\|_{\dot{F}^{t,q}_p}.
 \end{split}
\end{equation}
Equation~\eqref{eq.cont.tk.1} concludes the proof.
\end{proof}

Triebel-Lizorkin spaces offer a natural framework for atomic and molecular 
decompositions. Of course, setting $p = q = 2$ in the results above also yields
results for the more classical homogeneous Sobolev spaces $\dot{H}^\alpha$. 
Thus, Lemma~\ref{lemma.continuity.tk} tells us that operators of $\wsio$
continuously map $\dot{H}^t$ into $\dot{H}^{t+s}$ for $-s < t < 0$. 
In particular, this is valid for $s = 1/2$ and $t = -1/4$.

\subsection{Kernels defined on bounded domains}

Most results involving singular integral operators concern kernels defined on
the full space $\R^n \times \R^n$. Here, for finite time controllability, we 
need to adapt these results to a setting where the kernels are defined on 
squares, eg. $[0,1]\times[0,1]$. Atoms and molecules are localized functions.
Thus, it would be possible to carry on the same proof as above for bounded
domains, providing that the analogs of the representation 
lemmas~\ref{lemma.fjw.511} and ~\ref{lemma.fjw.518} exist for 
Triebel-Lizorkin spaces on bounded domains. In this paragraph, we give another
approach, which consists in proving that a kernel defined on a bounded domain
can be extended while satisfying the same estimates.

\begin{lemma} \label{lemma.extension.wsio}
Let $n = 1$, $0 < s < 1$ and $0 < \delta \leq 1$. Consider a kernel~$K$, defined
and continuous on $\Omega_1 = \left\{ (x,y) \in [0,1]^2, x \neq y \right\}$,
satisfying:
\begin{gather} 
  \label{wsio.01.i}
  \left| K(x,y) \right| \leq \kappa \left|x-y\right|^{-1+s}, \\
  \label{wsio.01.ii}
  \left| K(x',y) - K(x,y) \right| \leq \kappa \left|x'-x\right|^\delta
  \left|x-y\right|^{-1+s-\delta}, 
  \quad \textrm{for } \left|x'-x\right| \leq \frac{1}{2}\left|x-y\right|, \\
  \label{wsio.01.iii}
  \left| K(x,y') - K(x,y) \right| \leq \kappa \left|y'-y\right|^\delta
  \left|x-y\right|^{-1+s-\delta}, 
  \quad \textrm{for } \left|y'-y\right| \leq \frac{1}{2}\left|x-y\right|.
\end{gather}
Then there exists a kernel $\bar{K}$ on $\R\times\R$, continuous on
$\Omega$, such that:
\begin{itemize}
\item $\bar{K}$ is an extension of $K$: $\bar{K}\rvert_{\Omega_1} = K$,
\item 
  $\bar{K}$ is a weakly singular integral operator of type $(s,\delta)$
  on $\Omega$,
\item
  $\bar{K}$ is associated a constant $\kappa(\bar{K}) \leq C \kappa(K)$,
  where $C$ is independent of $K, s$ and $\delta$.
\end{itemize}
\end{lemma}

\begin{proof}
We start by defining $\bar{K}(x,y)$ on the infinite strip $-1 < y - x < 1$.
For $(x,y) \in \Omega_1$, we set $\bar{K}(x,y) = K(x,y)$. Outside
of the initial square, we extend by continuity the values taken on
the sides of the square and we choose an extension that is constant along 
all diagonal lines. Therefore, we define $\bar{K}(x,y)$ as:
\begin{equation} \label{kbar.strip}
 \begin{split}
   K(1+x-y,1) & \quad \textrm{for} \quad 1 \leq y, \enskip 0 < y - x < 1, \\
   K(0,y-x) & \quad \textrm{for} \quad x \leq 0, \enskip 0 < y - x < 1, \\
   K(1,1+y-x) & \quad \textrm{for} \quad 1 \leq x, \enskip 0 < x - y < 1, \\
   K(x-y,0) & \quad \textrm{for} \quad y \leq 0, \enskip 0 < x - y < 1.
 \end{split}
\end{equation}
Outside of the strip, we set:
\begin{equation} \label{kbar.decay}
 \begin{split}
  \bar{K}(x,y) = K(0,1) |x-y|^{-1+s}, \quad \textrm{for } y-x \geq 1, \\
  \bar{K}(x,y) = K(1,0) |x-y|^{-1+s}, \quad \textrm{for } x-y \geq 1.
 \end{split}
\end{equation}
This completes the definition of $\bar{K}$ on~$\Omega$. By construction, it is
easy to check that $\bar{K}$ is continuous on~$\Omega$. By construction, 
$\bar{K}$ also satisfies~\eqref{wsio.01.i} on $\Omega_1$, on the whole strip
$-1\leq y-x \leq 1$ thanks to~\eqref{kbar.decay} and on the half spaces 
$y-x\geq 1$ and $y-x\leq -1$ thanks to the decay chosen in~\eqref{kbar.decay}.

The Hölder regularity estimates~\eqref{wsio.01.ii} and~\eqref{wsio.01.iii} are a
little tougher. First, note that, by symmetry, one only needs to prove, for
example,~\eqref{wsio.01.ii} on the half place 
$\mathcal{H} = \{(x,y) \in \R \times \R, \quad y-x > 0\}$. We write 
$\mathcal{H} = \tilde{\mathcal{H}} \cup \mathcal{H}_1 \cup \mathcal{H}_{-} 
\cup \mathcal{H}_+$, where:
\begin{equation}
 \begin{split}
  \tilde{\mathcal{H}} & = \{(x,y) \in \mathcal{H}, \quad y-x > 1\}, \\
  \mathcal{H}_1 & = \{(x,y) \in \mathcal{H}, 
  	\quad 0 \leq x \textrm{ and } y \leq 1 \}, \\
  \mathcal{H}_+ & = \{(x,y) \in \mathcal{H}, 
  	\quad y-x \leq 1 \textrm{ and } 1 < y \}, \\
  \mathcal{H}_- & = \{(x,y) \in \mathcal{H}, 
  	\quad y-x \leq 1 \textrm{ and } x < 0 \}.
 \end{split} 
\end{equation}
Let $(x,y) \in \mathcal{H}$ and $(x',y)\in\mathcal{H}$ with 
$|x-x'|\leq \frac{1}{2}|x-y|$. If both points belong to the same subdomain, then
the Hölder regularity estimate in the $x$ direction for $\bar{K}$ is a direct 
consequence either of~\eqref{kbar.decay} on $\tilde{\mathcal{H}}$, 
of~\eqref{kbar.strip} on $\mathcal{H}_\pm$ and of the hypothesis on $K$ on 
$\mathcal{H}_1$. If the two points belong to different subdomains, we use a
triangular inequality involving a point at the boundary separating the two
subdomains. As an example of such a situation, if $x < 0 < x'$ and $y < x + 1$, 
then $(x,y) \in \mathcal{H}_-$ and $(x',y) \in \mathcal{H}_1$. We have:
\begin{equation}
 \begin{split}
  \left| \bar{K}(x,y) - \bar{K}(x',y) \right| 
  & = \left| K(0,y-x) - K(x',y) \right| \\
  & \leq \left|K(0,y-x) - K(0,y)\right| + \left|K(0,y) - K(x',y)\right| \\
  & \leq \kappa |x|^\delta|x-y|^{-1+s-\delta} 
  	+ \kappa |x'|^\delta |x'-y|^{-1+s-\delta} \\
  & \leq 5 \kappa |x-x'|^\delta|x-y|^{-1+s-\delta}.
 \end{split}
\end{equation}
The last inequality comes from the fact that $|x'|, |x| \leq |x-x'|$ and
$|x'-y|^{-1+s-\delta} \leq 4 |x-y|^{-1+s-\delta}$ for 
$|x-x'|\leq \frac{1}{2}|x-y|$. The details of the other situations are left to
the reader.
\end{proof}


\bibliographystyle{plain}
\bibliography{Burgers_Quadratic_Control}


\end{document}